\documentclass[11pt,reqno]{amsart}

\usepackage{tikz}
\usepackage{pgfplots}
\usepackage{comment}
\pgfplotsset{compat=1.14}
\usetikzlibrary{arrows}
\definecolor{ffvvqq}{rgb}{1,0.3333333333333333,0}
\definecolor{uuuuuu}{rgb}{0.26666666666666666,0.26666666666666666,0.26666666666666666}
\definecolor{ffvvqq}{rgb}{1,0.3333333333333333,0}
\definecolor{qqwuqq}{rgb}{0,0.39215686274509803,0}
\definecolor{fuqqzz}{rgb}{0.9568627450980393,0,0.6}
\definecolor{zzttff}{rgb}{0.6,0.2,1}
\definecolor{qqqqff}{rgb}{0,0,1}
\definecolor{qqccqq}{rgb}{0,0.8,0}

\usepackage{geometry}
\usepackage{graphicx}
\usepackage{amsmath,amsfonts,amsthm,amssymb}
\usepackage{color,soul}
\usepackage{graphicx}
\usepackage{verbatim}
\usepackage{color}
\usepackage{amscd}
\usepackage{verbatim}
\usepackage{mathrsfs}
\usepackage[colorlinks,citecolor=red,pagebackref,hypertexnames=false, breaklinks]{hyperref}
\usepackage{esint}
\usepackage{subcaption}

\begin{document}
\newcommand{\M}{{\mathcal M}}
\newcommand{\loc}{{\mathrm{loc}}}
\newcommand{\core}{C_0^{\infty}(\Omega)}
\newcommand{\sob}{W^{1,p}(\Omega)}
\newcommand{\sobloc}{W^{1,p}_{\mathrm{loc}}(\Omega)}
\newcommand{\merhav}{{\mathcal D}^{1,p}}
\newcommand{\be}{\begin{equation}}
\newcommand{\ee}{\end{equation}}
\newcommand{\mysection}[1]{\section{#1}\setcounter{equation}{0}}
\newcommand{\laplace}{\Delta}
\newcommand{\pl}{\laplace_p}
\newcommand{\grad}{\nabla}
\newcommand{\pd}{\partial}
\newcommand{\bo}{\pd}
\newcommand{\csub}{\subset \subset}
\newcommand{\sm}{\setminus}
\newcommand{\ssm}{:}
\newcommand{\diver}{\mathrm{div}\,}
\newcommand{\bea}{\begin{eqnarray}}
\newcommand{\eea}{\end{eqnarray}}
\newcommand{\bean}{\begin{eqnarray*}}
\newcommand{\eean}{\end{eqnarray*}}
\newcommand{\thkl}{\rule[-.5mm]{.3mm}{3mm}}
\newcommand{\cw}{\stackrel{\rightharpoonup}{\rightharpoonup}}
\newcommand{\id}{\operatorname{id}}
\newcommand{\supp}{\operatorname{supp}}
\newcommand{\calE}{\mathcal{E}}
\newcommand{\calF}{\mathcal{F}}
\newcommand{\wlim}{\mbox{ w-lim }}
\newcommand{\mymu}{{x_N^{-p_*}}}
\newcommand{\R}{{\mathbb R}}
\newcommand{\N}{{\mathbb N}}
\newcommand{\Z}{{\mathbb Z}}
\newcommand{\Q}{{\mathbb Q}}
\newcommand{\abs}[1]{\lvert#1\rvert}
\newtheorem{theorem}{Theorem}[section]
\newtheorem{corollary}[theorem]{Corollary}
\newtheorem{lemma}[theorem]{Lemma}
\newtheorem{notation}[theorem]{Notation}
\newtheorem{definition}[theorem]{Definition}
\newtheorem{remark}[theorem]{Remark}
\newtheorem{proposition}[theorem]{Proposition}
\newtheorem{assertion}[theorem]{Assertion}
\newtheorem{problem}[theorem]{Problem}
\newtheorem{conjecture}[theorem]{Conjecture}
\newtheorem{question}[theorem]{Question}
\newtheorem{example}[theorem]{Example}
\newtheorem{Thm}[theorem]{Theorem}
\newtheorem{Lem}[theorem]{Lemma}
\newtheorem{Pro}[theorem]{Proposition}
\newtheorem{Def}[theorem]{Definition}
\newtheorem{Exa}[theorem]{Example}
\newtheorem{Exs}[theorem]{Examples}
\newtheorem{Rems}[theorem]{Remarks}
\newtheorem{Rem}[theorem]{Remark}

\newtheorem{Cor}[theorem]{Corollary}
\newtheorem{Conj}[theorem]{Conjecture}
\newtheorem{Prob}[theorem]{Problem}
\newtheorem{Ques}[theorem]{Question}
\newtheorem*{corollary*}{Corollary}
\newtheorem*{theorem*}{Theorem}
\newcommand{\pf}{\noindent \mbox{{\bf Proof}: }}


\renewcommand{\theequation}{\thesection.\arabic{equation}}
\catcode`@=11 \@addtoreset{equation}{section} \catcode`@=12
\newcommand{\Real}{\mathbb{R}}
\newcommand{\real}{\mathbb{R}}
\newcommand{\Nat}{\mathbb{N}}
\newcommand{\ZZ}{\mathbb{Z}}
\newcommand{\CC}{\mathbb{C}}
\newcommand{\Pess}{\opname{Pess}}
\newcommand{\Proof}{\mbox{\noindent {\bf Proof} \hspace{2mm}}}
\newcommand{\mbinom}[2]{\left (\!\!{\renewcommand{\arraystretch}{0.5}
\mbox{$\begin{array}[c]{c}  #1\\ #2  \end{array}$}}\!\! \right )}
\newcommand{\brang}[1]{\langle #1 \rangle}
\newcommand{\vstrut}[1]{\rule{0mm}{#1mm}}
\newcommand{\rec}[1]{\frac{1}{#1}}
\newcommand{\set}[1]{\{#1\}}
\newcommand{\dist}[2]{$\mbox{\rm dist}\,(#1,#2)$}
\newcommand{\opname}[1]{\mbox{\rm #1}\,}
\newcommand{\mb}[1]{\;\mbox{ #1 }\;}
\newcommand{\undersym}[2]
 {{\renewcommand{\arraystretch}{0.5}  \mbox{$\begin{array}[t]{c}
 #1\\ #2  \end{array}$}}}
\newlength{\wex}  \newlength{\hex}
\newcommand{\understack}[3]{%
 \settowidth{\wex}{\mbox{$#3$}} \settoheight{\hex}{\mbox{$#1$}}
 \hspace{\wex}  \raisebox{-1.2\hex}{\makebox[-\wex][c]{$#2$}}
 \makebox[\wex][c]{$#1$}   }%
\newcommand{\smit}[1]{\mbox{\small \it #1}}
\newcommand{\lgit}[1]{\mbox{\large \it #1}}
\newcommand{\scts}[1]{\scriptstyle #1}
\newcommand{\scss}[1]{\scriptscriptstyle #1}
\newcommand{\txts}[1]{\textstyle #1}
\newcommand{\dsps}[1]{\displaystyle #1}
\newcommand{\dx}{\,\mathrm{d}x}
\newcommand{\dy}{\,\mathrm{d}y}
\newcommand{\dz}{\,\mathrm{d}z}
\newcommand{\dt}{\,\mathrm{d}t}
\newcommand{\dr}{\,\mathrm{d}r}
\newcommand{\du}{\,\mathrm{d}u}
\newcommand{\dv}{\,\mathrm{d}v}
\newcommand{\dV}{\,\mathrm{d}V}
\newcommand{\ds}{\,\mathrm{d}s}
\newcommand{\dS}{\,\mathrm{d}S}
\newcommand{\dk}{\,\mathrm{d}k}

\newcommand{\dphi}{\,\mathrm{d}\phi}
\newcommand{\dtau}{\,\mathrm{d}\tau}
\newcommand{\dxi}{\,\mathrm{d}\xi}
\newcommand{\deta}{\,\mathrm{d}\eta}
\newcommand{\dsigma}{\,\mathrm{d}\sigma}
\newcommand{\dtheta}{\,\mathrm{d}\theta}
\newcommand{\dnu}{\,\mathrm{d}\nu}
\newcommand{\Ker}{\mathrm{Ker}}
\newcommand{\Ima}{\mathrm{Im}}

\newcommand{\myset}[1]{\left\{#1\right\}}

\def\Xint#1{\mathchoice
{\XXint\displaystyle\textstyle{#1}}%
{\XXint\textstyle\scriptstyle{#1}}%
{\XXint\scriptstyle\scriptscriptstyle{#1}}%
{\XXint\scriptscriptstyle\scriptscriptstyle{#1}}%
\!\int}
\def\XXint#1#2#3{{\setbox0=\hbox{$#1{#2#3}{\int}$ }
\vcenter{\hbox{$#2#3$ }}\kern-.6\wd0}}
\def\dashint{\Xint-}

\newcommand{\Rd}{\color{red}}
\newcommand{\Bk}{\color{black}}
\newcommand{\Mg}{\color{magenta}}
\newcommand{\Wh}{\color{white}}
\newcommand{\Bl}{\color{blue}}
\newcommand{\Yl}{\color{yellow}}


\renewcommand{\div}{\mathrm{div}}
\newcommand{\red}[1]{{\color{red} #1}}

\newcommand{\cqfd}{\begin{flushright}                  
			 $\Box$
                 \end{flushright}}
                 
\newcommand{\todo}[1]{\vspace{5 mm}\par \noindent
\marginpar{\textsc{}} \framebox{\begin{minipage}[c]{0.95
\textwidth} \tt #1
\end{minipage}}\vspace{5 mm}\par}


\title{Reverse inequalities for quasi-Riesz transform on the Vicsek cable system}

\author{Baptiste Devyver and Emmanuel Russ}
\address{Baptiste Devyver, Institut Fourier - Universit\'e de Grenoble Alpes, France}
\email{baptiste.devyver@univ-grenoble-alpes.fr}
\address{Emmanuel Russ, Institut de Math\'ematiques de Marseille  - Aix-Marseille Universit\'e, France}
\email{emmanuel.russ@univ-amu.fr}

\maketitle
\tableofcontents

\begin{abstract}  
This work is devoted to the study of so-called ``reverse Riesz'' inequalities and suitable variants in the context of some fractal-like cable systems. It was already proved by L. Chen, T. Coulhon, J. Feneuil and the second author that, in the Vicsek cable system, the inequality $\left\Vert \Delta^{1/2}f\right\Vert_p\lesssim \left\Vert \nabla f\right\Vert_p$ is false for all $p\in [1,2)$. Following a recent joint paper by the two authors and M. Yang, we examine the validity of ``reverse quasi-Riesz'' inequalities, of the form $\left\Vert \Delta^{\gamma}e^{-\Delta}f\right\Vert_p\lesssim \left\Vert \nabla f\right\Vert_p$, in the (unbounded) Vicsek cable system, for $p\in (1,+\infty)$ and $\gamma>0$. These reverse inequalities are strongly related to the problem of $L^p$ boundedness of the operators $\nabla e^{-\Delta}\Delta^{-\varepsilon}$, the so-called ``quasi-Riesz transforms'' (at infinity), introduced by L. Chen in her PhD thesis. Our main result is an almost complete characterization of the sets of $\gamma\in (0,1)$ and $p\in (1,+\infty)$ such that the reverse quasi-Riesz inequality holds in the Vicsek cable system. It remains an open question to investigate reverse quasi-Riesz inequalities for other cable systems, or for manifolds built out of these. 
%
\end{abstract}

\section{Introduction}
Throughout the paper, if $A(f)$ and $B(f)$ are two nonnegative quantities defined for all $f$ belonging to a set $E$, the notation $A(f)\lesssim B(f)$ means that there exists $C>0$ such that $A(f)\le CB(f)$ for all $f\in E$, while $A(f)\simeq B(f)$ means that $A(f)\lesssim B(f)$ and $B(f)\lesssim A(f)$.\par
\noindent Let $M$ be a complete noncompact Riemannian manifold. Denote by $m$ the Riemannian measure, by $\nabla$ the Riemannian gradient and by $\Delta=-\mathrm{div}(\nabla\cdot)$ the (non-negative) Laplace-Beltrami operator. In this work, we consider the following three inequalities for $p\in (1,\infty)$ (where the $L^p$-norms are computed with respect to the measure $\mu$):

\begin{equation}\label{eq:Ep}\tag{$\mathrm{E}_p$}
||\Delta^{1/2}u||_p\lesssim ||\nabla u||_p\lesssim ||\Delta^{1/2}u||_p,\quad \forall u\in C_0^\infty(M),
\end{equation}

\begin{equation}\label{eq:Rp}\tag{$\mathrm{R}_p$}
||\nabla u||_p\lesssim ||\Delta^{1/2}u||_p,\quad \forall u\in C_0^\infty(M),
\end{equation}

\begin{equation}\label{eq:RRp}\tag{$\mathrm{RR}_p$}||\Delta^{1/2}u||_p\lesssim ||\nabla u||_p,\quad \forall u\in C_0^\infty(M).
\end{equation}
It follows easily from the Green formula and the self-adjointness of $\Delta$ that 

$$||\nabla u||_2^2=(\Delta u,u)=||\Delta^{1/2}||_2^2,\quad \forall u\in C_0^\infty(M).$$
Consequently, \eqref{eq:Ep} holds for $p=2$ on any complete Riemannian manifold. The inequality \eqref{eq:Rp} is equivalent to the $\Bk L^p$-boundedness of the Riesz transform $\mathscr{R}=\nabla\Delta^{-1/2}$.  For this reason, the inequality \eqref{eq:RRp} is called a ``reverse Riesz'' $L^p$ inequality. A well-known duality argument, originally introduced in \cite{B}, shows that \eqref{eq:Rp} implies ($\mathrm{RR}_q$) for $q=p'$ the conjugate exponent, but the converse implication does not hold. In order to present the known results concerning \eqref{eq:RRp}, we need to introduce some notations. We say that the measure $m$ is doubling provided

\begin{equation}\label{eq:D}\tag{D}
m(B(x,2r))\leq C m(B(x,r)),\quad \forall x\in M,\,\forall r>0.
\end{equation}
Let $p\in [1,\infty)$, we say that the scaled $L^p$ Poincar\'e inequality for balls holds, if for any ball $B=B(x_0,r)\subset M$, one has

\begin{equation}\label{eq:Pp}\tag{$\mathrm{P}_p$}
\int_{B}|f-f_B|^p\,dm\leq Cr^p\int_B|\nabla f|^p\,dm,\quad \forall f\in C^\infty(B),
\end{equation}
where $f_B=\frac{1}{m(B)}\int_B\int_Bf\,dm$. It has been shown in \cite{AC} that if \eqref{eq:D} and $\Rd(\mathrm{P}_q)\Bk$ hold for some $q\in [1,2)$, then \eqref{eq:RRp} holds for all $p\in (q,2)$. This result has been recently generalized by the two authors (\cite[Theorem 1.4]{DRGauss}): instead of \eqref{eq:Pp} for {\em all balls}, it is enough to assume it for all {\em remote} balls, that is balls which are far enough from a fixed reference point in $M$.

\medskip

In this work, we wish to consider reverse Riesz inequalities in a fractal-like setting. In her thesis \cite{chenthesis}, L. Chen has initiated the study of ``quasi-Riesz'' transforms, on manifolds whose heat kernel satisfy {\em sub-Gaussian} estimates for large times. Examples of such manifolds include manifolds which are built by fattening a graph with a fractal-like structure, such as the Vicsek graph. Later, in \cite{CCFR}, the authors have remarked that on such manifolds, certain reverse (quasi-)Riesz inequalities cannot hold. In particular, in the case of the Vicsek manifold, it is shown in \cite{CCFR} that even when a natural scaled $L^1$ Poincar\'e inequality on all balls holds (with a scaling which is not $r$ but is instead related naturally to the geometry of the manifold), the inequality \eqref{eq:RRp} is false {\it for all } $p\in (1,2)$. It thus appears that in fractal-like situations, the link between scaled Poincar\'e inequalities and reverse Riesz transforms is not clear. Our purpose in this work is to shed some light on this question, by investigating in details the case of the Vicsek cable system, a cable system built naturally from the Vicsek graph. More precisely, we will consider variants of \eqref{eq:RRp} where the power $\frac 12$ of $\Delta$ is replaced by other exponents related to the geometry of the cable system and to the parameter $\beta$ (the walk dimension) which appears in heat kernel estimates.

\subsection{The setting} \label{subsec:subgauss} Let us now introduce more precisely the setting. The following presentation is borrowed from \cite{DRY}. Let $(X,d,m,\calE,\calF)$ be an unbounded metric measure Dirichlet (MMD) space. Recall that $(X,d)$ is a locally compact separable unbounded metric space, $m$ is a positive Radon measure on $X$ with full support and $(\calE,\calF)$ is a strongly local regular Dirichlet form on $L^2(X;m)$. For all $x\in X$ and all $r\in(0,+\infty)$, let $B(x,r)=\{y\in X:d(x,y)<r\}$ be the open ball with center $x$ and radius $r$ and set $V(x,r):=m(B(x,r))$. Denote by $\Delta$ be the generator of the Dirichlet form $(\calE,\calF)$, by $(P_t)_{t>0}$ the corresponding heat semigroup and by $p_t$ the associated heat kernel. We also assume that the Dirichlet form admits a ``carr\'e du champ'' $\Gamma$, which allows us to consider the length of the gradient $|\nabla u|:=\Gamma(u)$ for suitable functions $u$. \par
\noindent Say that $X$ satisfies the doubling volume property if there exists $C>0$ such that, for all $x\in X$ and all $r>0$,
\begin{equation} \label{eq:DV} \tag{$D$}
V(x,2r)\le CV(x,r).
\end{equation}
\noindent Let $\beta\ge 2$, and let 

$$\Psi(r)=r^2\mathbf{1}_{(0,1)}(r)+r^\beta\mathbf{1}_{[1,+\infty)}(r).$$
Say that $X$ satisfies the sub-Gaussian heat kernel estimate with exponent $\beta$ if the heat kernel $p_t(x,y)$ satisfies the following estimate: there exist $C_1,C_2>0$ such that, for all $x,y \in X\setminus {\mathcal N}$ (where ${\mathcal N}$ is a properly exceptional set) and all $t>0$, 

\begin{equation}\label{eq:UHK}\tag{$UHK(\Psi)$}
p_{t}(x,y)\leq \frac{1}{V(x,\Psi^{-1}(C_1t))}\exp\left(-\Upsilon(C_2d(x,y),t)\right),
\end{equation}
where

$$\Upsilon(R,t)=\sup_{s>0}\left(\frac{R}{s}-\frac{t}{\Psi(s)}\right)\simeq \begin{cases}
\frac{R^2}{t}, & t<R\\ 
\left(\frac{R}{t^{1/\beta}}\right)^{\frac{\beta}{\beta-1}}, & t\geq R
\end{cases}$$
As a consequence of \eqref{eq:UHK} and elementary estimates (see e.g. \cite[p.9-10]{DRY}), the heat kernel satisfies the following upper-estimate:\Bk

\begin{equation}\label{B}\tag{$\mathrm{UE}_{\beta}$}
p_{t}(x,y)\leq\left\{ \begin{aligned}
& \frac{C}{V(x,\sqrt t)}\exp\left(-c\frac{d^{2}(x,y)}{t}\right), &0<t<1,\\
& \frac{C}{V(x,t^{1/\beta})}\exp\left(-c \left(\frac{d^{\beta}(x,y)}{t}\right)^{1/(\beta-1)}\right), &t\geq 1,
\end{aligned}\right.
\end{equation}
The estimates \eqref{eq:UHK} and \eqref{B} occur, in particular, on some fractal type manifolds or cable systems (see \cite[Appendix]{CCFR} and \cite{DRY}). It was established in \cite[Theorem 1.2]{CCFR} that, under the assumptions \eqref{eq:DV} and \eqref{B}, \eqref{eq:Rp} holds for all $p\in (1,2]$. As a consequence of duality, \eqref{eq:RRp} holds for all $p\in [2,\infty)$. In the particular case of the Vicsek manifold (\cite[Section 5]{CCFR}) shows that, under assumptions \eqref{eq:DV} and \eqref{B}, \eqref{eq:RRp} is false for {\em all} $p>2$, which entails that \eqref{eq:Rp} holds for no $p\in (1,2)$.\par

However, considering cable systems instead of Riemannian manifolds, we gave in \cite{DRY} sufficient conditions ensuring, in the presence of \eqref{eq:DV} and \eqref{B}, the $L^p$-boundedness of variants of the Riesz transforms (called ``quasi-Riesz transforms'' at infinity) for some values of $p\in (2,\infty)$. Namely, let $X$ be an unbounded cable system, $d$ be the metric on $X$ and $m$ be the measure on $X$. Assume that 
\begin{equation}\label{eq:VE2}\tag{$\mbox{V}_s$}
\frac{1}{C}\Phi(r)\le V(x,r)\le C\Phi(r)\text{ for any }x\in X,\text{ for any }r\in(0,+\infty),
\end{equation}
where, for some $\alpha\in [2,\beta+1]$, 
$$\Phi(r)=r1_{(0,1)}(r)+r^\alpha1_{[1,+\infty)}(r).$$
Assume that $p_t$ satisfies \eqref{B}. Assume also a ``generalized'' reverse H\"older inequality for the gradient of harmonic functions in balls, namely there exists $C_H\in(0,+\infty)$ such that, for all balls $B$ with radius $r$ and all functions $u$ harmonic in $2B$,
\begin{equation} \label{GRH}
\left\Vert \nabla u \right\Vert_{L^\infty(B)}\le C_H\frac{\Phi(r)}{\Psi(r)}\dashint_{2B} u,
\end{equation}
where
$$\Psi(r)=r1_{(0,1)}(r)+r^\beta1_{[1,+\infty)}(r).$$
\noindent Under these assumptions, \cite[Theorem 1.4]{DRY} states that the ``quasi-Riesz transform'' (at infinity) $\nabla e^{-\Delta}\Delta^{-\varepsilon}$ is $L^p$-bounded for all $p\in (1,\infty)$ and all $\varepsilon\in \left(0,1-\frac{\alpha}{\beta}\right)$. Let us mention that the motivation for introducing the extra term $e^{-\Delta}$ in the quasi-Riesz transform is to somehow neutralize the effect due to high frequencies in the spectrum of $\Delta$. See L. Chen's PhD thesis \cite{chenthesis} for additional details on the quasi-Riesz transform at infinity (in particular Theorem 1.5 therein). 

\subsection{The Vicsek cable system}
Let us recall how the Vicsek cable system is defined, following the presentation in \cite[Section 3]{DRY}. Let $N\ge2$ be an integer. In $\R^N$, let $p_1=(0,\ldots,0),\ldots,p_{2^N}$ be the vertices of the cube $[0,\frac{2}{\sqrt{N}}]^N\subseteq\R^N$, let $p_{2^N+1}=\frac{1}{2^N}\sum_{i=1}^{2^N}p_i=(\frac{1}{\sqrt{N}},\ldots,\frac{1}{\sqrt{N}})$. Let $f_i(x)=\frac{1}{3}x+\frac{2}{3}p_i$, $x\in\R^N$, $i=1,\ldots,2^N,2^N+1$. Then the $N$-dimensional Vicsek set is the unique non-empty compact set $K$ in $\R^N$ satisfying $K=\cup_{i=1}^{2^N+1}f_i(K)$.

Let $V_0=\{p_1,\ldots,p_{2^N},p_{2^N+1}\}$ and $V_{n+1}=\cup_{i=1}^{2^N+1}f_i(V_n)$ for any $n\ge0$. Then $\myset{V_n}_{n\ge0}$ is an increasing sequence of finite subsets of $K$ and the closure of $\cup_{n\ge0}V_n$ is $K$.

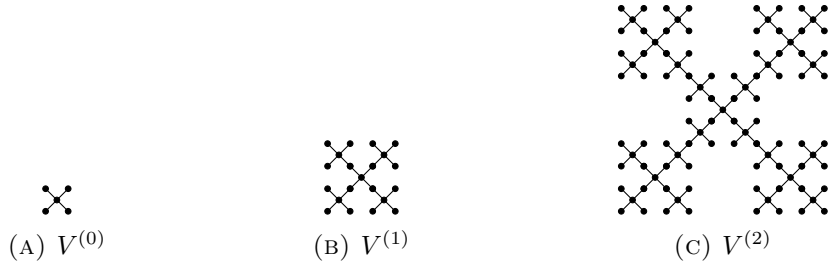
\begin{figure}[ht]
\centering
\begin{subfigure}[b]{0.2\textwidth}
\centering
\begin{tikzpicture}[scale=0.15]
\draw (0,0)--(2,2);
\draw (0,2)--(2,0);
\draw[fill=black] (0,0) circle (0.25);
\draw[fill=black] (0,2) circle (0.25);
\draw[fill=black] (2,2) circle (0.25);
\draw[fill=black] (2,0) circle (0.25);
\draw[fill=black] (1,1) circle (0.25);
\end{tikzpicture}
\caption{$V^{(0)}$}
\end{subfigure}
\hspace{2em}
\begin{subfigure}[b]{0.2\textwidth}
\centering
\begin{tikzpicture}[scale=0.15]
\draw (0,2)--(2,0);
\draw (0,0)--(2,2);
\draw (4,0)--(6,2);
\draw (4,2)--(6,0);
\draw (4,6)--(6,4);
\draw (4,4)--(6,6);
\draw (0,4)--(2,6);
\draw (2,4)--(0,6);
\draw (2,2)--(4,4);
\draw (2,4)--(4,2);

\draw[fill=black] (0,0) circle (0.25);
\draw[fill=black] (0,2) circle (0.25);
\draw[fill=black] (2,2) circle (0.25);
\draw[fill=black] (2,0) circle (0.25);
\draw[fill=black] (1,1) circle (0.25);

\draw[fill=black] (0+4,0) circle (0.25);
\draw[fill=black] (0+4,2) circle (0.25);
\draw[fill=black] (2+4,2) circle (0.25);
\draw[fill=black] (2+4,0) circle (0.25);
\draw[fill=black] (1+4,1) circle (0.25);

\draw[fill=black] (0,0+4) circle (0.25);
\draw[fill=black] (0,2+4) circle (0.25);
\draw[fill=black] (2,2+4) circle (0.25);
\draw[fill=black] (2,0+4) circle (0.25);
\draw[fill=black] (1,1+4) circle (0.25);

\draw[fill=black] (0+4,0+4) circle (0.25);
\draw[fill=black] (0+4,2+4) circle (0.25);
\draw[fill=black] (2+4,2+4) circle (0.25);
\draw[fill=black] (2+4,0+4) circle (0.25);
\draw[fill=black] (1+4,1+4) circle (0.25);

\draw[fill=black] (0+2,0+2) circle (0.25);
\draw[fill=black] (0+2,2+2) circle (0.25);
\draw[fill=black] (2+2,2+2) circle (0.25);
\draw[fill=black] (2+2,0+2) circle (0.25);
\draw[fill=black] (1+2,1+2) circle (0.25);
\end{tikzpicture}
\caption{$V^{(1)}$}
\end{subfigure}
\hspace{2em}
\begin{subfigure}[b]{0.3\textwidth}
\centering
\begin{tikzpicture}[scale=0.15]
\draw (0,2)--(2,0);
\draw (0,0)--(2,2);
\draw (4,0)--(6,2);
\draw (4,2)--(6,0);
\draw (4,6)--(6,4);
\draw (4,4)--(6,6);
\draw (0,4)--(2,6);
\draw (2,4)--(0,6);
\draw (2,2)--(4,4);
\draw (2,4)--(4,2);

\draw (0+12,2)--(2+12,0);
\draw (0+12,0)--(2+12,2);
\draw (4+12,0)--(6+12,2);
\draw (4+12,2)--(6+12,0);
\draw (4+12,6)--(6+12,4);
\draw (4+12,4)--(6+12,6);
\draw (0+12,4)--(2+12,6);
\draw (2+12,4)--(0+12,6);
\draw (2+12,2)--(4+12,4);
\draw (2+12,4)--(4+12,2);

\draw (0+6,2+6)--(2+6,0+6);
\draw (0+6,0+6)--(2+6,2+6);
\draw (4+6,0+6)--(6+6,2+6);
\draw (4+6,2+6)--(6+6,0+6);
\draw (4+6,6+6)--(6+6,4+6);
\draw (4+6,4+6)--(6+6,6+6);
\draw (0+6,4+6)--(2+6,6+6);
\draw (2+6,4+6)--(0+6,6+6);
\draw (2+6,2+6)--(4+6,4+6);
\draw (2+6,4+6)--(4+6,2+6);

\draw (0,2+12)--(2,0+12);
\draw (0,0+12)--(2,2+12);
\draw (4,0+12)--(6,2+12);
\draw (4,2+12)--(6,0+12);
\draw (4,6+12)--(6,4+12);
\draw (4,4+12)--(6,6+12);
\draw (0,4+12)--(2,6+12);
\draw (2,4+12)--(0,6+12);
\draw (2,2+12)--(4,4+12);
\draw (2,4+12)--(4,2+12);

\draw (0+12,2+12)--(2+12,0+12);
\draw (0+12,0+12)--(2+12,2+12);
\draw (4+12,0+12)--(6+12,2+12);
\draw (4+12,2+12)--(6+12,0+12);
\draw (4+12,6+12)--(6+12,4+12);
\draw (4+12,4+12)--(6+12,6+12);
\draw (0+12,4+12)--(2+12,6+12);
\draw (2+12,4+12)--(0+12,6+12);
\draw (2+12,2+12)--(4+12,4+12);
\draw (2+12,4+12)--(4+12,2+12);

\draw[fill=black] (0,0) circle (0.25);
\draw[fill=black] (0,2) circle (0.25);
\draw[fill=black] (2,2) circle (0.25);
\draw[fill=black] (2,0) circle (0.25);
\draw[fill=black] (1,1) circle (0.25);

\draw[fill=black] (0+4,0) circle (0.25);
\draw[fill=black] (0+4,2) circle (0.25);
\draw[fill=black] (2+4,2) circle (0.25);
\draw[fill=black] (2+4,0) circle (0.25);
\draw[fill=black] (1+4,1) circle (0.25);

\draw[fill=black] (0,0+4) circle (0.25);
\draw[fill=black] (0,2+4) circle (0.25);
\draw[fill=black] (2,2+4) circle (0.25);
\draw[fill=black] (2,0+4) circle (0.25);
\draw[fill=black] (1,1+4) circle (0.25);

\draw[fill=black] (0+4,0+4) circle (0.25);
\draw[fill=black] (0+4,2+4) circle (0.25);
\draw[fill=black] (2+4,2+4) circle (0.25);
\draw[fill=black] (2+4,0+4) circle (0.25);
\draw[fill=black] (1+4,1+4) circle (0.25);

\draw[fill=black] (0+2,0+2) circle (0.25);
\draw[fill=black] (0+2,2+2) circle (0.25);
\draw[fill=black] (2+2,2+2) circle (0.25);
\draw[fill=black] (2+2,0+2) circle (0.25);
\draw[fill=black] (1+2,1+2) circle (0.25);

\draw[fill=black] (0+12,0) circle (0.25);
\draw[fill=black] (0+12,2) circle (0.25);
\draw[fill=black] (2+12,2) circle (0.25);
\draw[fill=black] (2+12,0) circle (0.25);
\draw[fill=black] (1+12,1) circle (0.25);

\draw[fill=black] (0+4+12,0) circle (0.25);
\draw[fill=black] (0+4+12,2) circle (0.25);
\draw[fill=black] (2+4+12,2) circle (0.25);
\draw[fill=black] (2+4+12,0) circle (0.25);
\draw[fill=black] (1+4+12,1) circle (0.25);

\draw[fill=black] (0+12,0+4) circle (0.25);
\draw[fill=black] (0+12,2+4) circle (0.25);
\draw[fill=black] (2+12,2+4) circle (0.25);
\draw[fill=black] (2+12,0+4) circle (0.25);
\draw[fill=black] (1+12,1+4) circle (0.25);

\draw[fill=black] (0+4+12,0+4) circle (0.25);
\draw[fill=black] (0+4+12,2+4) circle (0.25);
\draw[fill=black] (2+4+12,2+4) circle (0.25);
\draw[fill=black] (2+4+12,0+4) circle (0.25);
\draw[fill=black] (1+4+12,1+4) circle (0.25);

\draw[fill=black] (0+2+12,0+2) circle (0.25);
\draw[fill=black] (0+2+12,2+2) circle (0.25);
\draw[fill=black] (2+2+12,2+2) circle (0.25);
\draw[fill=black] (2+2+12,0+2) circle (0.25);
\draw[fill=black] (1+2+12,1+2) circle (0.25);

\draw[fill=black] (0,0+12) circle (0.25);
\draw[fill=black] (0,2+12) circle (0.25);
\draw[fill=black] (2,2+12) circle (0.25);
\draw[fill=black] (2,0+12) circle (0.25);
\draw[fill=black] (1,1+12) circle (0.25);

\draw[fill=black] (0+4,0+12) circle (0.25);
\draw[fill=black] (0+4,2+12) circle (0.25);
\draw[fill=black] (2+4,2+12) circle (0.25);
\draw[fill=black] (2+4,0+12) circle (0.25);
\draw[fill=black] (1+4,1+12) circle (0.25);

\draw[fill=black] (0,0+4+12) circle (0.25);
\draw[fill=black] (0,2+4+12) circle (0.25);
\draw[fill=black] (2,2+4+12) circle (0.25);
\draw[fill=black] (2,0+4+12) circle (0.25);
\draw[fill=black] (1,1+4+12) circle (0.25);

\draw[fill=black] (0+4,0+4+12) circle (0.25);
\draw[fill=black] (0+4,2+4+12) circle (0.25);
\draw[fill=black] (2+4,2+4+12) circle (0.25);
\draw[fill=black] (2+4,0+4+12) circle (0.25);
\draw[fill=black] (1+4,1+4+12) circle (0.25);

\draw[fill=black] (0+2,0+2+12) circle (0.25);
\draw[fill=black] (0+2,2+2+12) circle (0.25);
\draw[fill=black] (2+2,2+2+12) circle (0.25);
\draw[fill=black] (2+2,0+2+12) circle (0.25);
\draw[fill=black] (1+2,1+2+12) circle (0.25);

\draw[fill=black] (0+12,0+12) circle (0.25);
\draw[fill=black] (0+12,2+12) circle (0.25);
\draw[fill=black] (2+12,2+12) circle (0.25);
\draw[fill=black] (2+12,0+12) circle (0.25);
\draw[fill=black] (1+12,1+12) circle (0.25);

\draw[fill=black] (0+4+12,0+12) circle (0.25);
\draw[fill=black] (0+4+12,2+12) circle (0.25);
\draw[fill=black] (2+4+12,2+12) circle (0.25);
\draw[fill=black] (2+4+12,0+12) circle (0.25);
\draw[fill=black] (1+4+12,1+12) circle (0.25);

\draw[fill=black] (0+12,0+4+12) circle (0.25);
\draw[fill=black] (0+12,2+4+12) circle (0.25);
\draw[fill=black] (2+12,2+4+12) circle (0.25);
\draw[fill=black] (2+12,0+4+12) circle (0.25);
\draw[fill=black] (1+12,1+4+12) circle (0.25);

\draw[fill=black] (0+4+12,0+4+12) circle (0.25);
\draw[fill=black] (0+4+12,2+4+12) circle (0.25);
\draw[fill=black] (2+4+12,2+4+12) circle (0.25);
\draw[fill=black] (2+4+12,0+4+12) circle (0.25);
\draw[fill=black] (1+4+12,1+4+12) circle (0.25);

\draw[fill=black] (0+2+12,0+2+12) circle (0.25);
\draw[fill=black] (0+2+12,2+2+12) circle (0.25);
\draw[fill=black] (2+2+12,2+2+12) circle (0.25);
\draw[fill=black] (2+2+12,0+2+12) circle (0.25);
\draw[fill=black] (1+2+12,1+2+12) circle (0.25);

\draw[fill=black] (0+6,0+6) circle (0.25);
\draw[fill=black] (0+6,2+6) circle (0.25);
\draw[fill=black] (2+6,2+6) circle (0.25);
\draw[fill=black] (2+6,0+6) circle (0.25);
\draw[fill=black] (1+6,1+6) circle (0.25);

\draw[fill=black] (0+4+6,0+6) circle (0.25);
\draw[fill=black] (0+4+6,2+6) circle (0.25);
\draw[fill=black] (2+4+6,2+6) circle (0.25);
\draw[fill=black] (2+4+6,0+6) circle (0.25);
\draw[fill=black] (1+4+6,1+6) circle (0.25);

\draw[fill=black] (0+6,0+4+6) circle (0.25);
\draw[fill=black] (0+6,2+4+6) circle (0.25);
\draw[fill=black] (2+6,2+4+6) circle (0.25);
\draw[fill=black] (2+6,0+4+6) circle (0.25);
\draw[fill=black] (1+6,1+4+6) circle (0.25);

\draw[fill=black] (0+4+6,0+4+6) circle (0.25);
\draw[fill=black] (0+4+6,2+4+6) circle (0.25);
\draw[fill=black] (2+4+6,2+4+6) circle (0.25);
\draw[fill=black] (2+4+6,0+4+6) circle (0.25);
\draw[fill=black] (1+4+6,1+4+6) circle (0.25);

\draw[fill=black] (0+2+6,0+2+6) circle (0.25);
\draw[fill=black] (0+2+6,2+2+6) circle (0.25);
\draw[fill=black] (2+2+6,2+2+6) circle (0.25);
\draw[fill=black] (2+2+6,0+2+6) circle (0.25);
\draw[fill=black] (1+2+6,1+2+6) circle (0.25);

\end{tikzpicture}
\caption{$V^{(2)}$}
\end{subfigure}
\caption{$V^{(0)}$, $V^{(1)}$ and $V^{(2)}$ for $N=2$}\label{fig_V012_Vicsek}
\end{figure}

For any $n\ge0$, let $V^{(n)}=3^nV_n=\{3^nv:v\in V_n\}$, see Figure \ref{fig_V012_Vicsek} for $V^{(0)}$, $V^{(1)}$ and $V^{(2)}$ for $N=2$. Then $\myset{V^{(n)}}_{n\ge0}$ is an increasing sequence of finite sets. Let $V=\cup_{n\ge0}V^{(n)}$ and $E=\{\{p,q\}:p,q\in V,|p-q|=1\}$, then $(V,E)$ is an infinite, locally bounded, connected graph, the corresponding unbounded cable system is called the $N$-dimensional Vicsek cable system. Each closed (open) cable is a(n) closed (open) interval in $\R^N$ and the cable system $X$ is defined as
$$X=\bigcup_{\mbox{\tiny
$
\begin{subarray}{c}
p,q\in V\\
|p-q|=1
\end{subarray}
$
}}
[p,q]\subseteq\R^N,$$
here $[p,q]$ denotes the closed interval with endpoints $p,q\in\R^N$.  Let $m$ be the unique positive Radon measure on $X$ satisfying
$m([a,b])=|a-b|$ for all $a;b\in X$ belonging to the same closed cable. It was proved in (\cite[Equation (4.14)]{BCG01}) that \eqref{eq:VE2} holds with $\alpha=\log(2^N+1)/\log3$. \par
\noindent The Dirichlet form on $X$ is defined by

$$\mathcal{E}(u,u)=\sum_{\{p,q\}\in E}\int_{(p,q)}|\nabla u|^2\,dm,$$
where $\nabla u$ denotes the one-dimensional gradient of $u$ in each open cable $(p,q)$. Note that $\nabla u$ is actually well-defined everywhere, except at the vertices points; but since $m(V)=0$, it follows that $\nabla u$ is defined almost everywhere. Then, by definition, for any $p\in [1,+\infty]$,

$$||\nabla u||_{L^p(X,m)}:=||\nabla u||_{L^p(X\setminus V,m)}.$$

\begin{definition}
{\em 
For any $n\ge0$, we say that a subset $W$ of $X$ is an {\em $n$-skeleton} if $W$ is a translation of the intersection of the closed convex hull of $V^{(n)}$ and $X$. It is obvious that the closed convex hull of $W$ is a cube, we say that the $2^N$ vertices of the cube are the boundary points of the skeleton and the center of the cube is the center of the skeleton.
}
\end{definition}

\subsection{Statements of our results}
The following general but simple lemma shows that the boundedness of the quasi-Riesz transform becomes stronger when the value of $\varepsilon$ decreases:
\begin{Lem}\label{lem:var-eps}

Let $(X,d,m,\mathcal{E},\mathcal{F})$ be a MMD space with a ``carr\'e du champ''. Let $\varepsilon\in (0,1)$ and $p\in (1,+\infty)$ be such that the quasi-Riesz transform $\nabla e^{-\Delta}\Delta^{-\varepsilon}$ is bounded on $L^p$. Then, for every $0<s\leq \varepsilon$, the quasi-Riesz transform $\nabla e^{-\Delta}\Delta^{-s}$ is bounded on $L^p$ as well.

\end{Lem}
As a corollary of Lemma \ref{lem:var-eps} and of the results of \cite{CCFR}, one gets:

\begin{Cor}\label{cor:p12}

Let $(X,d,m,\mathcal{E},\mathcal{F})$ be a cable system. Assume the volume growth condition \eqref{eq:VE2} and the heat kernel sub-Gaussian estimate \eqref{B}. Then, for every $\varepsilon\in (0,\frac{1}{2}]$ and $p\in (1,2]$, the quasi-Riesz transform $\nabla e^{-\Delta}\Delta^{-\varepsilon}$ is bounded on $L^p$.

\end{Cor}
As for the usual Riesz transform, by a duality argument, the $L^p$-boundedness of the quasi-Riesz transform implies a ``reverse inequality'':
\begin{Lem}\label{lem:RiesztoRRp}
Let $\varepsilon\in (0,1)$, $p\in (1,\infty)$ and $p^{\prime}$ be such that $\frac 1p+\frac 1{p^{\prime}}=1$. Then the $L^{p^{\prime}}$-boundedness of $\nabla e^{-\Delta}\Delta^{-\varepsilon}$ implies the following reverse inequality:
$$
\left\Vert \Delta^{1-\varepsilon}e^{-\Delta}f\right\Vert_p\lesssim \left\Vert \nabla f\right\Vert_p.
$$
\end{Lem}
This motivates us to introduce the following {\em reverse quasi-Riesz inequalities}: if $\gamma>0$, consider the inequality

\begin{equation} \label{eq:quasiRRp}\tag{$\mathrm{RR}_{p,\gamma}$}
\left\Vert \Delta^{\gamma}e^{-\Delta}f\right\Vert_p\lesssim \left\Vert \nabla f\right\Vert_p.
\end{equation}
Corollary \ref{cor:p12} and Lemma \ref{lem:RiesztoRRp} immediately yield:

\begin{Cor}

Let $(X,d,m,\mathcal{E},\mathcal{F})$ be an unbounded cable system. Assume the volume growth condition \eqref{eq:VE2} and the heat kernel sub-Gaussian estimate \eqref{B}. Then, for any $p\in [2,\infty]$ and $\gamma\in [\frac{1}{2},1)$, the reverse quasi-Riesz inequality \eqref{eq:quasiRRp} holds.

\end{Cor}

\begin{question}
{\em 
Given the results of Lemma \ref{lem:var-eps} and Lemma \ref{lem:RiesztoRRp}, it is reasonable to expect that \eqref{eq:quasiRRp} becomes stronger when $\gamma$ decreases. However, so far we have been unable to prove this in full generality. We leave this as an interesting open question.
}
\end{question}

Lemma \ref{lem:RiesztoRRp} and \cite[Theorem 1.4]{DRY} imply at once that, if $(X,d,m,\calE,\calF)$ is an unbounded cable system such that \eqref{eq:VE2}, \eqref{B} and \eqref{GRH} hold, then, for all $p\in (1,\infty)$ and all $\gamma\in (\frac{\alpha}{\beta},1)$, \eqref{eq:quasiRRp} holds.  It is an open problem in \cite{DRY} to establish the $L^p$-boundedess of the quasi-Riesz transform $\nabla\Delta^{-\varepsilon} e^{-\Delta}$ at the threshold $\varepsilon=1-\frac{\alpha}{\beta}$, so Lemma \ref{lem:RiesztoRRp} does not easily imply any reverse inequality for the Vicsek cable system if $\gamma= \frac{\alpha}{\beta}$. Likewise, nothing seems to be known concerning the case $\gamma<\frac{\alpha}{\beta}$. On the other hand, for $p=2$ one can easily prove the following universal result, valid in any MMD space:

\begin{Lem}\label{lem:RR2}

Let $(X,d,m,\mathcal{E},\mathcal{F})$ be a MMD space with a ``carr\'e du champ''. Let $\varepsilon\in (0,1)$. The quasi-Riesz transform $\nabla e^{-\Delta}\Delta^{-\varepsilon}$ is bounded on $L^2$, if and only if $\varepsilon \in (0,\frac{1}{2}]$. Also, if this is the case then the reverse inequality {\em ($\mathrm{RR}_{2,1-\varepsilon}$)} holds, and

\begin{equation}\label{eq:E2}\tag{$\mathrm{E}_{2,\varepsilon}$}
\left\Vert \Delta^{1-\varepsilon}e^{-\Delta}f\right\Vert_2\lesssim  \left\Vert \nabla f\right\Vert_2 \lesssim ||\Delta^{\varepsilon}e^{-\Delta} f||_2.
\end{equation}

\end{Lem}

\medskip

Note that for the Vicsek cable system, $\beta=\alpha+1$ and $\alpha>1$, so that $1-\frac{\alpha}{\beta}=\frac{1}{\alpha+1}<\frac{1}{2}$. Hence, there is a ``gap'' between the results for $p=2$ and for $p\neq 2$: namely, for $p=2$, $(\mathrm{RR}_{2,\gamma})$ holds for $\gamma\in \left[\frac{1}{2},1\right]$, while nothing is known about the validity of \eqref{eq:quasiRRp} for $\gamma\in [\frac{1}{2},\frac{\alpha}{\beta}]$ and $p\neq 2$. As has been mentioned previously in the introduction, in the case of manifolds with heat kernel Gaussian estimates and satisfying scaled Poincar\'e inequalities for geodesic balls, reverse Riesz inequalities have been investigated in \cite{AC} and more recently in \cite{DRGauss}; however, there are some real difficulties in adapting their arguments to the sub-Gaussian situation: for instance, the proofs in these papers use in a crucial way the fact that the scaling in the Poincar\'e inequalities \eqref{eq:Pp} is exactly $r^p$, $r$ being the radius of the ball (in particular to establish a Calder\'on-Zygmund decomposition for Sobolev spaces). This behaviour is false even for the special case of the Vicsek cable system: in fact, in this case the correct scaling is $r^{\alpha+p-1}$, see \cite[Section 5]{chenthesis}. Let us also mention that Poincar\'e inequalities have recently been investigated in some fractal situations in \cite{BC}. The purpose of this article is to elucidate part of this problem: here we focus on the Vicsek cable system itself and we characterize in an almost optimal way the set of parameters $\gamma$ and $p$ such that the reverse quasi-Riesz inequality \eqref{eq:quasiRRp} holds. The main result of this paper writes as follows:

\begin{Thm}\label{thm:RRp-subg}
Consider the Vicsek cable system. Let $\gamma\in (0,1)$ and $p\in (1,+\infty)$. Define $p^*=p^*(\gamma)$ as follows:
$$p^*=\begin{cases}
\frac{\alpha-1}{\gamma(\alpha+1)-1}& \mbox{ if }\gamma\in (\frac{1}{\alpha+1},\frac{\alpha}{\alpha+1}),\\
+\infty& \mbox{ if }\gamma\leq \frac{1}{\alpha+1},\\
1& \mbox{ if }\gamma\geq \frac{\alpha}{\alpha+1}.
\end{cases}$$
Then, the reverse inequality \eqref{eq:quasiRRp} holds whenever $\gamma\in[\frac{1}{2},1)$ and $p>p^*$, and is false whenever $\gamma\in (0,1)$ and $p< p^*$.
\end{Thm}

\begin{Rem} \label{open}
{\em 
It is of course an interesting open problem to determine if  \eqref{eq:quasiRRp} holds at the threshold $p=p^*$ if $p^*>1$. The answer is unclear at the moment. Likewise, it would be interesting to know whether \eqref{eq:quasiRRp} holds for $p>p^*$ and $\gamma<\frac{1}{2}$; our method does not allow us to conclude in this case but we conjecture that the result still holds. A related open question concerns the boundedness of the quasi-Riesz transform: is $\nabla e^{-\Delta}\Delta^{-\varepsilon}$ bounded on $L^p$ for $\varepsilon\in \left.\left(0, \frac{\alpha}{\alpha+1}\right]\right.$ and $1<p<(p^{\ast}(1-\varepsilon))^{\prime}=\frac{\alpha-1}{\varepsilon(\alpha+1)-1}$? Note that by duality this would imply \eqref{eq:quasiRRp}, $\gamma=1-\varepsilon$ for $p>p^{\ast}(\gamma)$, hence it would recover part of the result of Theorem \ref{thm:RRp-subg}.
}
\end{Rem}
\Bk 

\begin{Rem}
{\em It was established in \cite[Theorem 1.4]{Badr} that, when the underlying space satisfies the doubling volume property and the scaled $L^{p_0}$ Poincar\'e inequalities (with the standard scaling), then, for $p<q$, homogeneous Sobolev spaces $\dot{W}^{1,p}$ and $\dot{W}^{1,q}$ interpolate by the real interpolation method. It is stated as an open question in \cite[Section 1, p. 237]{Badr} whether this interpolation result for homogeneous Sobolev spaces still holds without the Poincar\'e inequality hypothesis; here, we see that the Vicsek cable system provides an example of a space for which the interpolation result of \cite{Badr} is not applicable. Indeed, if the conclusion of \cite[Theorem 1.4]{Badr} was true in this case, then from Lemma \ref{lem:RR2}, Theorem \ref{thm:RRp-subg} and equation \eqref{eq:interpolate}, one would conclude that \eqref{eq:quasiRRp} holds for every $p\in (1,\infty)$ and every $\gamma\in (\frac{1}{2},\frac{\alpha}{\alpha+1})$. However, the result of Theorem \ref{thm:RRp-subg} tells us that this is not the case.}
\end{Rem}

Finally, we point out that the behaviour of the quasi-Riesz transforms and the reverse inequalities are quite different in a Gaussian context; in this direction, one has the following result:

\begin{Pro}\label{pro:RRmfd}
Assume that $M$ is a Riemannian manifold satisfying \eqref{eq:DV} as well as a {\it Gaussian } pointwise upper bound for the heat kernel. Assume moreover $V(o,r)\simeq r^D$ for all $r\geq 1$ and some $o\in M$. Then, for all $\gamma\in \left(0,\frac{1}{2}\right)$ and all $p\in (1,\infty)$, the reverse inequality \eqref{eq:quasiRRp} is false. 
\end{Pro}

\begin{Rem}
{\em 
Assume that $M$ is a Riemannian manifold fulfilling the hypotheses of Proposition \ref{pro:RRmfd}; assume moreover that the Riesz transform $\nabla \Delta^{-1/2}$ on $M$ is bounded on $L^p$ for every $p\in (1,+\infty)$. Then, Lemma \ref{lem:var-eps} and Proposition \ref{pro:RRmfd} imply that the quasi-Riesz transform $\nabla e^{-\Delta}\Delta^{-\varepsilon}$ is bounded on $L^p$ (resp. the reverse inequality \eqref{eq:quasiRRp} holds), if and only if $\varepsilon\in \left.\left(0,\frac{1}{2}\right]\right.$ (resp. $\gamma\in \left.\left[\frac{1}{2},1\right)\right.$) and $p\in (1,+\infty)$. This is quite different from the conjectural picture for reverse quasi-Riesz inequalities described in Remark \ref{open} for the Vicsek cable system.
}
\end{Rem}

\noindent {\bf Idea of the proof of Theorem \ref{thm:RRp-subg}:} we first prove a suitable Poincar\'e inequality for skeletons in the Vicsek cable system (see Lemma \ref{lem:Poinc}). This is the main ingredient at the technical level, and the main novelty of this paper. Then, we show that it implies a modified Calder\'on-Zygmund decomposition in Sobolev spaces (Lemma \ref{lem:CZ2}). Once this is done, we follow the approach of P. Auscher and T. Coulhon in \cite{AC} and prove a weak reverse quasi-Riesz inequality (Lemma \ref{lem:weak}). The positive result in Theorem \ref{thm:RRp-subg} follows by interpolation arguments. The negative result in Theorem \ref{thm:RRp-subg} is proved by contradiction extending arguments from \cite{CCFR}: more precisely, the contradiction is obtained by combining Nash inequality and the assumed reverse quasi-Riesz inequality.

\section{$L^p$ reverse quasi-Riesz inequalities in the Vicsek cable system} \label{sec:subgauss}
\subsection{Some elementary facts about ``quasi-Riesz'' inequalities} \label{subsec:revQR}

\begin{proof}[Proof of Lemma \ref{lem:var-eps}]

Let $0<s< \varepsilon$, and write

$$\nabla e^{-\Delta}\Delta^{-s}=c\int_0^\infty \nabla e^{-(t+1)\Delta} \frac{dt}{t^{1-s}}.$$
The boundedness of the quasi-Riesz transform $\nabla e^{-\Delta}\Delta^{-\varepsilon}$ on $L^p$, together with the $L^p$-analyticity of the heat semi-group, imply that

$$||\nabla e^{-(t+1)\Delta}||_{p\to p}\lesssim (1+t)^{-\varepsilon},$$
so the integral 

$$\int_0^\infty ||\nabla e^{-(t+1)\Delta}||_{p\to p} \frac{dt}{t^{1-s}}$$
converges.

\end{proof}

\begin{proof} [Proof of Lemma \ref{lem:RiesztoRRp}]
To prove \eqref{eq:quasiRRp}, let $h\in L^{p^{\prime}}\cap {\mathcal D}_{p^{\prime}}(\Delta)$ such that $\left\Vert \Delta^{\varepsilon}h\right\Vert_{L^{p^{\prime}}}\le 1$. Then
\begin{eqnarray*}
\left\vert \langle \Delta^{1-\varepsilon}e^{-\Delta}f,\Delta^{\varepsilon}h\rangle\right\vert & = & \left\vert \langle f,\Delta e^{-\Delta}h\rangle\right\vert \\
& = & \left\vert \langle \nabla f,\nabla e^{-\Delta}h\rangle\right\vert\\
& \le & \left\Vert \nabla f\right\Vert_p \left\Vert \nabla e^{-\Delta}h\right\Vert_{p^{\prime}}\\
& \lesssim & \left\Vert \nabla f\right\Vert_p \left\Vert \Delta^{\varepsilon}h\right\Vert_{p^{\prime}}\\
& \le & \left\Vert \nabla f\right\Vert_p.
\end{eqnarray*}
Since the semigroup generated by $\Delta$ is self-adjoint on $L^2(X)$ and is a contraction on $L^r(X)$ for all $r\in [1,+\infty]$, so is the semigroup generated by $\Delta^{\varepsilon}$ (\cite[Proposition 13.1]{SSV}). Therefore, \cite[Lemma 1]{R} entails that, for all $q\in (1,\infty)$, $\left\{\Delta^{\varepsilon}h;\ h\in L^q(X)\cap {\mathcal D}_q(\Delta)\right\}$ is dense in $L^q(X)$. This and the previous calculation prove that \eqref{eq:quasiRRp} holds.

\end{proof}

%
%
%
Thanks to Lemma \ref{lem:RiesztoRRp} and \cite[Theorem 1.4]{DRY} (which relies on the gradient heat kernel estimate obtained in this paper), we obtain:

\begin{Cor}\label{cor:RRpesp}
On the Vicsek cable system, the reverse inequalities \eqref{eq:quasiRRp} holds for any $p\in (1,+\infty)$ and $\gamma\in (\frac{\alpha}{\beta},1)$.
\end{Cor}
\medskip

\begin{proof}[Proof of Lemma \ref{lem:RR2}]
Let us first assume that $0<\varepsilon\leq \frac{1}{2}$, and let us prove \eqref{eq:quasiRRp} for $p=2$. The case $\varepsilon=\frac{1}{2}$ follows from the facts that $e^{-\Delta}$ is bounded on $L^2$ and the Riesz transform $\nabla \Delta^{-1/2}$ is bounded --in fact, an isometry-- on $L^2$ (this latter fact comes from the definition of the Laplacian and its relationship to the quadratic form). In the case $\varepsilon<\frac{1}{2}$, Lemma \ref{lem:var-eps} shows that the quasi-Riesz transform $\nabla e^{-\Delta}\Delta^{-\varepsilon}$ is bounded on $ L^2$. The reverse inequalities follow directly from the boundedness of the quasi-Riesz transforms and Lemma \ref{lem:RiesztoRRp}.

Conversely, assume that the quasi-Riesz transform $\nabla e^{-\Delta}\Delta^{-\varepsilon}$ is bounded on $L^2$ for some $\varepsilon>0$. Equivalently, assume the inequality

$$||\nabla e^{-\Delta} f||_2\lesssim ||\Delta^{\varepsilon}f||_2$$
for every $f$ in the domain of the quadratic form associated to $\Delta$. By the Green formula,

$$||\nabla e^{-\Delta}f||_2=||\Delta^{1/2}e^{-\Delta}f||_2,$$
hence

$$||\Delta^{1/2}e^{-\Delta}f||_2\lesssim ||\Delta^{\varepsilon}f||_2.$$
Using the Spectral Theorem, this implies the following inequality for all $x\geq 0$:

$$\sqrt{x}e^{-x}\lesssim x^\varepsilon .$$
Looking at the behaviour for $x\to 0$, we conclude that $\varepsilon \leq \frac{1}{2}$.
%
%

\end{proof}
Another result, proved with arguments similar to previous ones, and that will be useful later, is the following boundedness result for ``small times'':

\begin{Lem}\label{lem:small_time}

Consider the Vicsek cable system; let $p\in (1,+\infty)$, and denote $q=p'$. Let $\gamma \in (0,1)$. For $r>0$, denote 

$$T_r=\int_0^{ r\Bk}\frac{\partial}{\partial t} e^{-(t+1)\Delta}\,\frac{dt}{t^{ \gamma }}.$$
Let $r_0>0$. Then, there exists a constant $C=C(r_0,p)>0$ such that, for every $0\leq r\leq r_0$, 

\begin{equation}\label{eq:small}
||T_rf||_p\leq C||\nabla f||_p,\quad f\in C_0^\infty(X).
\end{equation}

\end{Lem}

\begin{Rem}
{\em 
As the proof shows, the result of Lemma \ref{lem:small_time} holds for more general MMD spaces,  under the assumption that sub-gaussian heat kernel gradient estimates for small times hold.
}

\end{Rem}

\begin{proof}

Introduce the following conjugate operator:

$$Q_r:=\int_0^r e^{-(t+1)\Delta}\,\frac{dt}{t^{ \gamma }}.$$
By the gradient heat kernel for small times on the Vicsek cable system (see \cite[Theorem 1.1]{DRY}), one has

$$||\nabla e^{-(t+1)\Delta}||_{q\to q}\leq C,\quad \forall t\leq r_0.$$
This implies that there is a constant $C>0$ such that

$$||\nabla Q_r||_{q\to q} \leq C,\quad \forall r\leq r_0.$$
Note that

$$T_r=\Delta\int_0^{r} e^{-(t+1)\Delta}\,\frac{dt}{t^{\gamma }}=\Delta Q_r=Q_r\Delta.$$
Let $g\in L^q$, then one has

\bean
\langle T_rf,g\rangle & = & \langle f,\Delta Q_rg\rangle \\
& = & \langle \nabla f,\nabla Q_rg\rangle \\
&\leq & ||\nabla f||_p\cdot ||\nabla Q_r g||_q\\
&\leq & C||\nabla f||_p\cdot ||g||_q.
\eean
This implies by $L^p-L^q$ duality that

$$ ||T_rf||_p\leq C||\nabla f||_p,$$
hence the result.

\end{proof}

\subsection{More on the Vicsek cable system}

In this subsection, we collect a few facts about the Vicsek cable system, which will be useful for the rest of the paper. For $n\in \mathbb{N}$, denote by $\mathscr{V}_n$ the set of all $n$-skeletons in $X$. The construction of the Vicsek cable system entails the following facts:

\begin{itemize}

\item $X$ is a tree, i.e. does not contain any closed loop. As a consequence, if $x\in X$ and $A\subset X$ is closed, and $L=d(x,A)$, then there is a unique geodesic $\gamma:[0,L]\to X$ such that $\gamma(0)=x$, $\gamma(L)\in A$.

\item Let $n\in\mathbb{N}$, then

$$X=\bigcup_{V_n\in\mathscr{V}_n}V_n.$$

\item Any $n$-skeleton is a ball of radius $3^{n}$, whose volume is equal to $4\times 5^{n}$.

\item Let $n\in\mathbb{N}$ and let $V_n$ be an $n$-skeleton. Then, there exists a unique $(n+1)$-skeleton $V_{n+1}$ such that $V_n\subset V_{n+1}$. The unique $(n+1)$-skeleton $V_{n+1}$ will be called the skeleton ``daughter'' of the skeleton $V_n$, and $V_n$ will be called one of the $5$ ``parents'' of $V_{n+1}$.

\item Let $V$ and $W$ be two skeletons such that $V\cap W\neq \emptyset$. Then, either $V\subset W$, or $W\subset V$, or $V\cap W$ consists precisely of one point, which is one of the four boundary points of both $V$ and $W$.

\end{itemize}
We will use the following simple lemma on balls in the Vicsek cable system, which is based on the particular structure of the cable system and the fact that $n$-skeletons have diameters equal to $2\cdot 3^{n}$ :

\begin{Lem}\label{lem:skel_balls}

Let $x\in X$, $r>0$ and $n\in\N^*$. 

\begin{itemize}

\item[(i)] if $r> 2\cdot 3^{n}$ then there exists an $n$-skeleton inside the ball $B(x,r)$.

\item[(ii)] if $r\leq 2\cdot 3^{n}$  then the ball $\overline{B(x,r)}$ is contained inside some $(n+1)$-skeleton. 

\end{itemize}
Consequently, if $r>2$, then letting $n=\lfloor \log_3\left(\frac{r}{2}\right)\rfloor$, the ball $B(x,r)$ contains an $n$-skeleton $V_n$ such that $\mathrm{diam}(V_n)\simeq r$ and $m(V_n)\simeq r^\alpha$. 

\end{Lem}

\begin{definition} \label{def:diag}
Let $V$ be a skeleton of $X$.
\begin{enumerate}
\item Denote by $\mathrm{diag}(V)$ the union of the two great diagonals of $V$, which are the diagonals of the closed convex hull of $V$.
\item For $p\in V\setminus \mathrm{diag}(V)$, we will denote $\pi(p)$ be the ``projection'' of $p$ onto the diagonals of $V$, more precisely $\pi(p)$ is the unique vertex point in $\mathrm{diag}(V)$ achieving $\mathrm{dist}(p,\mathrm{diag}(V))$. The uniqueness of $\pi(p)$ comes from the first remark above ($X$ is a tree). For $p\in \mathrm{diag}(V)$, we simply let $\pi(p)=p$. 
\item We say that a function $\varphi : V\to \R$ is {\em radial}, if for any $p\in V$,

$$\varphi(p)=\varphi\circ \pi(p).$$

\end{enumerate}
\end{definition}
This alternative notion of ``radial'' function will turn out to be better suited to the particular structure of the Vicsek cable system, than the usual notion of radial functions.

\subsection{A Poincar\'e inequality for skeletons}

The key technical point for the proof of Theorem \ref{thm:RRp-subg} is a Poincar\'e inequality on skeletons, contained in the following lemma:

\begin{Lem}\label{lem:Poinc}

Let $q\in [1,+\infty)$. There exists a constant $C=C(q)>0$ such that the following holds: for every $n\in\N^*$, every $n$-skeleton $V_n\in\mathscr{V}_n$ and every $f\in C^\infty(V_n)$, denoting 

$$c_n(f):=\frac{1}{m(\mathrm{diag}(V_n))}\int_{\mathrm{diag}(V_n)}f\,dm,$$
then one has

\begin{equation}\label{eq:P1}
\int_{V_n}|f-c_n(f)|^q\,dm \leq Cm(V_n)^{1+\frac{q-1}{\alpha}}\int_{V_n}|\nabla f|^q\,dm
\end{equation}
and

\begin{equation}\label{eq:P2}
\int_{\mathrm{diag}(V_n)}|f-c_n(f)|^q\,dm \leq C\mathrm{diam}(V_n)^q\int_{V_n}|\nabla f|^q\,dm.
\end{equation}

\end{Lem}

\begin{Rem}
{\em 
In fact, as the proof shows, it is possible to show the stronger inequality:

$$\int_{\mathrm{diag}(V_n)}|f-c_n(f)|^q\,dm \leq C\mathrm{diam}(V_n)^q\int_{\mathrm{diag}(V_n)}|\nabla f|^q\,dm.$$
However, the inequality \eqref{eq:P2} will suffice for the purposes of this article.

}

\end{Rem}

\begin{Rem}
{\em 
It is important to notice that one can substract {\em the same constant }$c_n$ both in \eqref{eq:P1} and in \eqref{eq:P2}!
}
\end{Rem}

\begin{proof}

Notice that without loss of generality, we can --and will-- assume that $c_n(f)=0$. Since by assumption $c_n(f)=0$ and $f$ is continuous, there exists $z\in\mathrm{diag}(V_n)$ such that $f(z)=0$. We introduce some notations. Denote by $(\pi_i)_{i=1,\cdots,k}$, where $k=2\mathrm{diam}(V_n)-1$, the vertices points of $V_n$ that lie on the diagonals $\mathrm{diag}(V_n)$. The unique shortest path joining two points $p$ and $s$ in the cable system will be denoted $[p,s]$. For $i=1,\cdots,k$, we let $\Gamma_i$ denote the set of points $p$ in $V_n\setminus \mathrm{diag}(V_n)$  such that $\pi(p)=\pi_i$. With these notations settled, we start the proof of \eqref{eq:P1}.  For $p,s\in V_n$, we have by H\"older,

\begin{eqnarray}
|f(p)-f(s)|^q &\leq & \left( \int_{[p,s]}|\nabla f|\,dm\right)^q  \\
&\leq & m([p,s])^{q-1}\int_{[p,s]}|\nabla f|^q\,dm \\
&\leq & m([p,s])^{q-1}\int_{V_n}|\nabla f|^q\,dm\\
\label{eq:distance}
&\leq & \mathrm{diam}(V_n)^{q-1}\int_{V_n}|\nabla f|^q\,dm.
\end{eqnarray}
First, choosing $s=z$, and integrating the above inequality in $p\in\mathrm{diag}(V_n)$ yields

\begin{equation}\label{eq:P3}
\int_{\mathrm{diag}(V_n)}|f(p)|^q\,dm(p)\leq m(\mathrm{diag}(V_n))\times \mathrm{diam}(V_n)^{q-1}\int_{V_n}|\nabla f|^q\,dm,
\end{equation}
which, given that $m(\mathrm{diag}(V_n))\times \mathrm{diam}(V_n)^{q-1}=2 \mathrm{diam}(V_n)^{q}$ and recalling that $c_n=0$, implies \eqref{eq:P2} with constant $C=2$.

On the other hand, if $\pi(p)=\pi_i$ then by \eqref{eq:distance},

$$|f(p)-f(\pi_i)|^q\leq  \mathrm{diam}(V_n)^{q-1}\int_{V_n}|\nabla f|^q\,dm,$$
therefore by integration, for every $i=1,\cdots,k$,

\bean
\int_{\Gamma_i}|f(p)-f(\pi_i)|^q\,dm(p) &\leq &  m(\Gamma_i)\times \mathrm{diam}(V_n)^{q-1}\int_{V_n}|\nabla f|^q\,dm.
\eean
Since $\Gamma_i\cap\Gamma_j=\emptyset$ for $i\neq j$, and $\cup_{i}\Gamma_i=V_n\setminus \mathrm{diag}(V_n)$, by summing the preceeding inequality over all $i=1,\cdots,k$, we get

$$\int_{V_n\setminus\mathrm{diag}(V_n)}|f(p)-f(\pi(p))|^q\,dm(p) \leq m(V_n) \times \mathrm{diam}(V_n)^{q-1}\int_{V_n}|\nabla f|^q\,dm.$$
Consequently, the inequality $(a+b)^q\leq 2^{q-1}(a^q+b^q)$ for all $a,b\geq0$ yields

\bean
\int_{V_n\setminus \mathrm{diag}(V_n)}|f(p)|^q\,dm(p) &\lesssim & \sum_{i=1}^k  m(\Gamma_i)|f(\pi_i)|^q+\int_{V_n\setminus \mathrm{diag}(V_n)}|f(p)-f(\pi(p))|^q\,dm(p)\\
&\lesssim & \sum_{i=1}^k  m(\Gamma_i) |f(\pi_i)|^q+ m(V_n)\times \mathrm{diam}(V_n)^{q-1}\int_{V_n} |\nabla f|^q\,dm.
\eean
Using \eqref{eq:P3}, and the trivial inequality $m(\mathrm{diag}(V_n))\leq m(V_n)$, we get

$$\int_{V_n}|f(p)|^q\,dm(p) \lesssim \sum_{i=1}^k  m(\Gamma_i) |f(\pi_i)|^q+ m(V_n)\times \mathrm{diam}(V_n)^{q-1}\int_{V_n} |\nabla f|^q\,dm.$$
Since $f(z)=0$, we now estimate

\bean
\sum_{i=1}^km(\Gamma_i)|f(\pi_i)|^q &\lesssim &  \sum_{i=1}^km(\Gamma_i)|f(\pi_i)-f(z)|^q \\
&\leq & \left(\sum_{i=1}^km(\Gamma_i)\right) \mathrm{diam}(V_n)^{q-1}\int_{V_n}|\nabla f|^q\,dm \\
& = &  m(V_n)\times \mathrm{diam}(V_n)^{q-1}\int_{V_n}|\nabla f|^q\,dm,
\eean
where \eqref{eq:distance} has been used to pass from the first to the second line. Thus, one gets

$$\int_{V_n}|f(p)|^q\,dm(p)\lesssim  m(V_n)\times \mathrm{diam}(V_n)^{q-1}\int_{V_n}|\nabla f|^q\,dm.$$
Now we notice that $\mathrm{diam}(V_n)^\alpha\simeq m(V_n)$ (see Lemma \ref{lem:skel_balls}), so the above inequality implies 

$$\int_{V_n}|f(p)|^q\,dm(p)\lesssim  m(V_n)^{1+\frac{q-1}{\alpha}}\int_{V_n}|\nabla f|^q\,dm.$$
This is precisely \eqref{eq:P1}.

\end{proof}
\subsection{Extension of Poincar\'e inequalities to more general subsets of $X$}
The above proof actually works for much more general subsets of $X$ than skeletons. To explain this, we first introduce a few definitions. 
\begin{definition} \label{def:soul}
Let $\Gamma\subset X$ be a closed set.
\begin{enumerate}
\item Since $X$ is a tree, there is a well-defined ``projection map'' $\pi_\Gamma : X\to \Gamma$, generalizing Definition \ref{def:diag}; explicitly, $\pi_\Gamma(p)$ is defined to be the unique point of $\Gamma$ achieving $d(p,\Gamma)$. 
\item Let $Y\subset X$. Say that $Y$ is {\em geodesically convex} if and only if, for any two points $x$ and $y$ in $Y$, the (unique) geodesic between $x$ and $y$ lies inside $Y$. We note that since $X$ is a tree, a subset $Y$ of $X$ is connected, if and only if it is geodesically convex.
\item Let $A\subset X$ be a connected subset, such that $\Gamma\subset \bar{A}$; for $p\in \Gamma$ a vertex point, we let 
$$\Gamma_p^A:=\{q\in A\setminus\Gamma \,;\,\pi_\Gamma(q)=p\}.$$
Roughly speaking, the set $\Gamma_p^A$ is the union of all the ``branches'' emanating from $p$ in $A$. Since $A$ is connected, for every $p$ vertex in $\Gamma$, the set $\Gamma_p^A\subset A$ is connected, hence geodesically convex. 
\item Let $A\subset X$ be connected. The closed set $\Gamma$ is called a {\em soul} of $A$ if the following conditions are satisfied:
\begin{enumerate}
\item $\Gamma\subset \bar{A}$.
\item $\Gamma$ is geodesically convex.
\end{enumerate}
Note that, as a consequence of the uniqueness of the projection onto $\Gamma$, i.e. the fact that $\pi_\Gamma$ is well-defined, the sets $\Gamma_p^A$, for $p$ vertex in $\Gamma$, which are not empty are disjoint.
\end{enumerate}
\end{definition}
\noindent Skeletons are examples of sets admitting a soul: take for $\Gamma$ the union of the two diagonals. The following remark will be useful:

\begin{Lem}\label{lem:soul}

Let $B\subset A\subset X$ be two connected sets, and let $\Gamma\subset A$ be a soul for $A$. Assume that $B\cap \Gamma\neq \emptyset$. Then, $\overline{B\cap \Gamma}$ is a soul for $B$, and $\pi_\Gamma(B)=\Gamma\cap B$.

\end{Lem}
\begin{proof}

Since the intersection of two geodesically convex sets is still geodesically convex, and since $B$ and $\Gamma$ are geodesically convex, we deduce that $B\cap \Gamma$ is geodesically convex, and in turn that $\overline{B\cap \Gamma}$ is geodesically convex, too. Fix $z\in\Gamma\cap B$, and let $x\in B$. Recall that $\pi_\Gamma$ denotes the ``projection'' onto $\Gamma$ in $A$. Let $y=\pi_\Gamma(x)$. We claim that $y\in B$. Indeed, the segment $[z,x]$ lies inside $B$ by convexity, and by uniqueness of $y$ there must be that $y\in [z,x]$. Therefore, one deduces that $\pi_\Gamma(B)=\Gamma\cap B$.



\end{proof}
We can now state the following extension of Lemma \ref{lem:Poinc},  the proof of which is analogous to the one of Lemma \ref{lem:Poinc}:

\begin{Lem}\label{lem:Poinc2}

Let $A\subset X$ be connected, and $q\in [1,+\infty)$. Assume that one can find a soul $\Gamma$ of $A$, such that $\pi_\Gamma(A)=\Gamma\cap A$, and such that for some constant $c>0$, 

$$c^{-1}m(\Gamma)\leq \mathrm{diam}(A)\leq c m(A)^{\frac{1}{\alpha}}.$$
Denote, for every $f\in C^\infty(A)$,

$$(f):=\frac{1}{m(\Gamma)}\int_\Gamma f(p)\,dm(p).$$
Then, there is a constant $C=C(c,q)>0$ such that

\begin{equation}\label{eq:P21}
\int_{A}|f-c(f)|^q\,dm \leq Cm(V_n)^{1+\frac{q-1}{\alpha}}\int_{A}|\nabla f|^q\,dm
\end{equation}
and

\begin{equation}\label{eq:P22}
\int_{\Gamma}|f-c(f)|^q\,dm \leq C\mathrm{diam}(V_n)^q\int_{A}|\nabla f|^q\,dm.
\end{equation}

\end{Lem}
It will be important for what follows to be able to construct simple souls for sets which are finite unions of skeletons. The setting is the following: take a connected set $A\subset X$ which can be written as a finite union:

$$A=\bigcup_{k=1}^\ell W_k,$$
where the $W_k$'s are pairwise distinct $n$-skeletons, 
Since the skeletons $W_k$ are pairwise distinct, they have disjoint interior. Consider $\Gamma$ to be the closed set, consisting of the union of the diagonals of the skeletons $W_k$, namely (see figure \ref{fig_union}):

$$\Gamma:=\bigcup_{k=1}^\ell \mathrm{diag}(W_k).$$

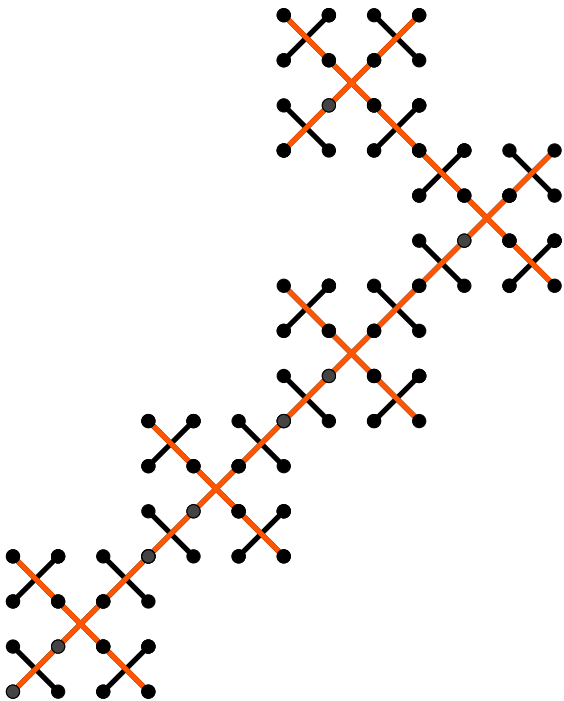
\begin{figure}[ht]
\centering
\begin{tikzpicture}[line cap=round,line join=round,>=triangle 45,x=0.6cm,y=0.6cm]
\clip(-14,-10) rectangle (14,8);
\draw [line width=2pt] (0,1)-- (1,0);
\draw [line width=2pt] (1,1)-- (0,0);
\draw [line width=2pt] (1,2)-- (2,1);
\draw [line width=2pt] (2,2)-- (1,1);
\draw [line width=2pt] (2,3)-- (3,2);
\draw [line width=2pt] (3,3)-- (2,2);
\draw [line width=2pt] (0,3)-- (1,2);
\draw [line width=2pt] (1,3)-- (0,2);
\draw [line width=2pt] (2,1)-- (3,0);
\draw [line width=2pt] (3,1)-- (2,0);
\draw [line width=2pt] (-3,-2)-- (-2,-3);
\draw [line width=2pt] (-2,-2)-- (-3,-3);
\draw [line width=2pt] (-2,-1)-- (-1,-2);
\draw [line width=2pt] (-1,-1)-- (-2,-2);
\draw [line width=2pt] (-1,0)-- (0,-1);
\draw [line width=2pt] (0,0)-- (-1,-1);
\draw [line width=2pt] (-3,0)-- (-2,-1);
\draw [line width=2pt] (-2,0)-- (-3,-1);
\draw [line width=2pt] (-1,-2)-- (0,-3);
\draw [line width=2pt] (0,-2)-- (-1,-3);
\draw [line width=2pt,color=ffvvqq] (-3,0)-- (0,-3);
\draw [line width=2pt,color=ffvvqq] (-3,-3)-- (0,0);
\draw [line width=2pt,color=ffvvqq] (0,0)-- (3,3);
\draw [line width=2pt,color=ffvvqq] (0,3)-- (3,0);
\draw [line width=2pt] (-6,-5)-- (-5,-6);
\draw [line width=2pt] (-5,-5)-- (-6,-6);
\draw [line width=2pt] (-5,-4)-- (-4,-5);
\draw [line width=2pt] (-4,-4)-- (-5,-5);
\draw [line width=2pt] (-4,-3)-- (-3,-4);
\draw [line width=2pt] (-3,-3)-- (-4,-4);
\draw [line width=2pt] (-6,-3)-- (-5,-4);
\draw [line width=2pt] (-5,-3)-- (-6,-4);
\draw [line width=2pt] (-4,-5)-- (-3,-6);
\draw [line width=2pt] (-3,-5)-- (-4,-6);
\draw [line width=2pt,color=ffvvqq] (-6,-3)-- (-3,-6);
\draw [line width=2pt,color=ffvvqq] (-6,-6)-- (-3,-3);
\draw [line width=2pt] (-3,4)-- (-2,3);
\draw [line width=2pt] (-2,4)-- (-3,3);
\draw [line width=2pt] (-2,5)-- (-1,4);
\draw [line width=2pt] (-1,5)-- (-2,4);
\draw [line width=2pt] (-1,6)-- (0,5);
\draw [line width=2pt] (0,6)-- (-1,5);
\draw [line width=2pt] (-3,6)-- (-2,5);
\draw [line width=2pt] (-2,6)-- (-3,5);
\draw [line width=2pt] (-1,4)-- (0,3);
\draw [line width=2pt] (0,4)-- (-1,3);
\draw [line width=2pt,color=ffvvqq] (-3,3)-- (0,6);
\draw [line width=2pt,color=ffvvqq] (-3,6)-- (0,3);
\draw [line width=2pt] (-9,-8)-- (-8,-9);
\draw [line width=2pt] (-8,-8)-- (-9,-9);
\draw [line width=2pt] (-8,-7)-- (-7,-8);
\draw [line width=2pt] (-7,-7)-- (-8,-8);
\draw [line width=2pt] (-7,-6)-- (-6,-7);
\draw [line width=2pt] (-6,-6)-- (-7,-7);
\draw [line width=2pt] (-9,-6)-- (-8,-7);
\draw [line width=2pt] (-8,-6)-- (-9,-7);
\draw [line width=2pt] (-7,-8)-- (-6,-9);
\draw [line width=2pt] (-6,-8)-- (-7,-9);
\draw [line width=2pt,color=ffvvqq] (-9,-6)-- (-6,-9);
\draw [line width=2pt,color=ffvvqq] (-9,-9)-- (-6,-6);
\begin{scriptsize}
\draw [fill=uuuuuu] (0,0) circle (2.5pt);
\draw [fill=black] (1,0) circle (2.5pt);
\draw [fill=black] (1,1) circle (2.5pt);
\draw [fill=black] (0,1) circle (2.5pt);
\draw [fill=black] (1,2) circle (2.5pt);
\draw [fill=black] (2,1) circle (2.5pt);
\draw [fill=black] (2,2) circle (2.5pt);
\draw [fill=uuuuuu] (1,1) circle (2.5pt);
\draw [fill=black] (2,3) circle (2.5pt);
\draw [fill=black] (3,2) circle (2.5pt);
\draw [fill=black] (3,3) circle (2.5pt);
\draw [fill=uuuuuu] (2,2) circle (2.5pt);
\draw [fill=black] (2,2) circle (2.5pt);
\draw [fill=black] (0,3) circle (2.5pt);
\draw [fill=black] (1,2) circle (2.5pt);
\draw [fill=black] (1,3) circle (2.5pt);
\draw [fill=uuuuuu] (0,2) circle (2.5pt);
\draw [fill=black] (0,2) circle (2.5pt);
\draw [fill=uuuuuu] (1,3) circle (2.5pt);
\draw [fill=black] (1,3) circle (2.5pt);
\draw [fill=black] (2,1) circle (2.5pt);
\draw [fill=black] (3,0) circle (2.5pt);
\draw [fill=black] (3,1) circle (2.5pt);
\draw [fill=uuuuuu] (2,0) circle (2.5pt);
\draw [fill=black] (2,0) circle (2.5pt);
\draw [fill=uuuuuu] (3,1) circle (2.5pt);
\draw [fill=black] (3,1) circle (2.5pt);
\draw [fill=black] (2,1) circle (2.5pt);
\draw [fill=black] (-3,-2) circle (2.5pt);
\draw [fill=black] (-2,-3) circle (2.5pt);
\draw [fill=black] (-2,-2) circle (2.5pt);
\draw [fill=uuuuuu] (-3,-3) circle (2.5pt);
\draw [fill=black] (-2,-1) circle (2.5pt);
\draw [fill=black] (-1,-2) circle (2.5pt);
\draw [fill=black] (-1,-1) circle (2.5pt);
\draw [fill=uuuuuu] (-2,-2) circle (2.5pt);
\draw [fill=black] (-1,0) circle (2.5pt);
\draw [fill=black] (0,-1) circle (2.5pt);
\draw [fill=black] (0,0) circle (2.5pt);
\draw [fill=uuuuuu] (-1,-1) circle (2.5pt);
\draw [fill=black] (-3,0) circle (2.5pt);
\draw [fill=black] (-2,-1) circle (2.5pt);
\draw [fill=black] (-2,0) circle (2.5pt);
\draw [fill=uuuuuu] (-3,-1) circle (2.5pt);
\draw [fill=black] (-1,-2) circle (2.5pt);
\draw [fill=black] (0,-3) circle (2.5pt);
\draw [fill=black] (0,-2) circle (2.5pt);
\draw [fill=uuuuuu] (-1,-3) circle (2.5pt);
\draw [fill=black] (-1,-1) circle (2.5pt);
\draw [fill=black] (-3,-1) circle (2.5pt);
\draw [fill=uuuuuu] (-2,0) circle (2.5pt);
\draw [fill=black] (-2,0) circle (2.5pt);
\draw [fill=black] (-1,-3) circle (2.5pt);
\draw [fill=uuuuuu] (0,-2) circle (2.5pt);
\draw [fill=black] (0,-2) circle (2.5pt);
\draw [fill=black] (-1,-2) circle (2.5pt);
\draw [fill=black] (-6,-5) circle (2.5pt);
\draw [fill=black] (-5,-6) circle (2.5pt);
\draw [fill=black] (-5,-5) circle (2.5pt);
\draw [fill=uuuuuu] (-6,-6) circle (2.5pt);
\draw [fill=black] (-5,-4) circle (2.5pt);
\draw [fill=black] (-4,-5) circle (2.5pt);
\draw [fill=black] (-4,-4) circle (2.5pt);
\draw [fill=uuuuuu] (-5,-5) circle (2.5pt);
\draw [fill=black] (-4,-3) circle (2.5pt);
\draw [fill=black] (-3,-4) circle (2.5pt);
\draw [fill=black] (-3,-3) circle (2.5pt);
\draw [fill=uuuuuu] (-4,-4) circle (2.5pt);
\draw [fill=black] (-6,-3) circle (2.5pt);
\draw [fill=black] (-5,-4) circle (2.5pt);
\draw [fill=black] (-5,-3) circle (2.5pt);
\draw [fill=uuuuuu] (-6,-4) circle (2.5pt);
\draw [fill=black] (-4,-5) circle (2.5pt);
\draw [fill=black] (-3,-6) circle (2.5pt);
\draw [fill=black] (-3,-5) circle (2.5pt);
\draw [fill=uuuuuu] (-4,-6) circle (2.5pt);
\draw [fill=black] (-6,-3) circle (2.5pt);
\draw [fill=black] (-3,-6) circle (2.5pt);
\draw [fill=uuuuuu] (-6,-6) circle (2.5pt);
\draw [fill=uuuuuu] (-3,-3) circle (2.5pt);
\draw [fill=black] (-4,-4) circle (2.5pt);
\draw [fill=black] (-6,-4) circle (2.5pt);
\draw [fill=uuuuuu] (-5,-3) circle (2.5pt);
\draw [fill=black] (-5,-3) circle (2.5pt);
\draw [fill=black] (-4,-6) circle (2.5pt);
\draw [fill=uuuuuu] (-3,-5) circle (2.5pt);
\draw [fill=black] (-3,-5) circle (2.5pt);
\draw [fill=black] (-4,-5) circle (2.5pt);
\draw [fill=black] (-3,4) circle (2.5pt);
\draw [fill=black] (-2,3) circle (2.5pt);
\draw [fill=black] (-2,4) circle (2.5pt);
\draw [fill=uuuuuu] (-3,3) circle (2.5pt);
\draw [fill=black] (-2,5) circle (2.5pt);
\draw [fill=black] (-1,4) circle (2.5pt);
\draw [fill=black] (-1,5) circle (2.5pt);
\draw [fill=uuuuuu] (-2,4) circle (2.5pt);
\draw [fill=black] (-1,6) circle (2.5pt);
\draw [fill=black] (0,5) circle (2.5pt);
\draw [fill=black] (0,6) circle (2.5pt);
\draw [fill=uuuuuu] (-1,5) circle (2.5pt);
\draw [fill=black] (-3,6) circle (2.5pt);
\draw [fill=black] (-2,5) circle (2.5pt);
\draw [fill=black] (-2,6) circle (2.5pt);
\draw [fill=uuuuuu] (-3,5) circle (2.5pt);
\draw [fill=black] (-1,4) circle (2.5pt);
\draw [fill=black] (0,3) circle (2.5pt);
\draw [fill=black] (0,4) circle (2.5pt);
\draw [fill=uuuuuu] (-1,3) circle (2.5pt);
\draw [fill=uuuuuu] (-3,3) circle (2.5pt);
\draw [fill=black] (0,6) circle (2.5pt);
\draw [fill=black] (-3,6) circle (2.5pt);
\draw [fill=black] (0,3) circle (2.5pt);
\draw [fill=black] (-1,5) circle (2.5pt);
\draw [fill=black] (-3,5) circle (2.5pt);
\draw [fill=uuuuuu] (-2,6) circle (2.5pt);
\draw [fill=black] (-2,6) circle (2.5pt);
\draw [fill=black] (-1,3) circle (2.5pt);
\draw [fill=uuuuuu] (0,4) circle (2.5pt);
\draw [fill=black] (0,4) circle (2.5pt);
\draw [fill=black] (-1,4) circle (2.5pt);
\draw [fill=black] (-3,3) circle (2.5pt);
\draw [fill=black] (-9,-8) circle (2.5pt);
\draw [fill=black] (-8,-9) circle (2.5pt);
\draw [fill=black] (-8,-8) circle (2.5pt);
\draw [fill=uuuuuu] (-9,-9) circle (2.5pt);
\draw [fill=black] (-8,-7) circle (2.5pt);
\draw [fill=black] (-7,-8) circle (2.5pt);
\draw [fill=black] (-7,-7) circle (2.5pt);
\draw [fill=uuuuuu] (-8,-8) circle (2.5pt);
\draw [fill=black] (-7,-6) circle (2.5pt);
\draw [fill=black] (-6,-7) circle (2.5pt);
\draw [fill=black] (-6,-6) circle (2.5pt);
\draw [fill=uuuuuu] (-7,-7) circle (2.5pt);
\draw [fill=black] (-9,-6) circle (2.5pt);
\draw [fill=black] (-8,-7) circle (2.5pt);
\draw [fill=black] (-8,-6) circle (2.5pt);
\draw [fill=uuuuuu] (-9,-7) circle (2.5pt);
\draw [fill=black] (-7,-8) circle (2.5pt);
\draw [fill=black] (-6,-9) circle (2.5pt);
\draw [fill=black] (-6,-8) circle (2.5pt);
\draw [fill=uuuuuu] (-7,-9) circle (2.5pt);
\draw [fill=black] (-9,-6) circle (2.5pt);
\draw [fill=black] (-6,-9) circle (2.5pt);
\draw [fill=uuuuuu] (-9,-9) circle (2.5pt);
\draw [fill=uuuuuu] (-6,-6) circle (2.5pt);
\draw [fill=uuuuuu] (-6,-6) circle (2.5pt);
\draw [fill=black] (-7,-7) circle (2.5pt);
\draw [fill=black] (-9,-7) circle (2.5pt);
\draw [fill=uuuuuu] (-8,-6) circle (2.5pt);
\draw [fill=black] (-8,-6) circle (2.5pt);
\draw [fill=black] (-7,-9) circle (2.5pt);
\draw [fill=uuuuuu] (-6,-8) circle (2.5pt);
\draw [fill=black] (-6,-8) circle (2.5pt);
\draw [fill=black] (-7,-8) circle (2.5pt);
\end{scriptsize}
\end{tikzpicture}
\caption{The set $A$ with its soul $\Gamma$ (in red)}\label{fig_union}
\end{figure}

Since $A$ is connected, $\Gamma$ is geodesically convex. It follows that $\Gamma$ is a soul for $A$.

Let us record this simple fact as a lemma for future use:

\begin{Lem}\label{lem:soul2}
Let $A\subset X$ be a connected subset of $X$, which can be written as
$$A=\bigcup_{k=1}^\ell W_k,$$
where the $W_k$'s are pairwise distinct $n$-skeletons. Consider $\Gamma$ to be the closed set, consisting of the union of the diagonals of the skeletons $W_j$. Then, $\Gamma$ is a soul for $A$.
\end{Lem}
Finally, we consider the following setting: we let $A\subset X$ be open and connected, and we assume that there is a covering

$$A=\bigcup_{i\in \mathbb{N}} B_i,$$
where each $B_i$ is a ball $B_i=B(x_i,r_i)$ with $r_i>8$. Furthermore, we assume that:

\begin{itemize}

\item[(a)] the covering is locally finite: there is a constant $N\in\mathbb{N}^*$ such that for every ball $B_i$, the set $J_i$ of $j$'s such that $B_i\cap B_j\neq \emptyset$ is finite, and $|J_i|\leq N$.

\item[(b)] there is a constant $c\geq 1$ such that for any two intersecting balls $B_i$ and $B_j$, there holds:

$$c^{-1}r_j\leq r_i\leq c r_j.$$

\end{itemize}
We intend to construct for any $i\in\mathbb{N}$ a special soul $\Gamma_i$ for $B_i$; these souls will then be called the {\em souls adapted to the covering}.

The procedure is as follows: fix a ball $B_i$ of the covering; given property (b) of the covering and thanks to Lemma \ref{lem:skel_balls}, one can find $n_i\in \mathbb{N}$ depending only on $r_i$ and on $c$, such that $B_i$, as well as every ball $B_j$, $j\in J_i$, contains an $n_i$-skeleton $W_j$, and such that moreover, for some constant $C>0$ depending only on $c$,

$$C^{-1} m(B_j)\leq m(W_j)\leq Cm(B_j),\quad C^{-1}\mathrm{diam}(W_j)\leq r_j\leq C\mathrm{diam}(W_j).$$
Explicitly, one may take

$$n_i=\max\left(\lfloor \log_3\left(\frac{r_i}{8c}\right)\rfloor,0\right).$$
As will prove useful later, choosing this particular value of $n_i$, one sees that it can moreover be assumed that the diameter of the $n_i$-skeletons $W_j$ satisfies:
$$\mathrm{diam}(W_j)=2\cdot 3^{n_i}<\frac{r_i}{4}.$$
\Bk
Furthermore, the ball $B_i$ can be covered by a finite number $M\in\mathbb{N}^*$ of $n_i$-skeletons (which includes in particular the skeleton $W_i$), where $M$ depends only on $n_i$. Let $S$ be the union of all these $n_i$-skeletons intersecting $B_i$, so that $B_i\subset S$, and let $\Gamma$ be the union of the diagonals of all these $n_i$-skeletons. 

According to Lemma \ref{lem:soul2}, $\Gamma$ is a soul for $S$. Since $B_i$ contains at least one such $n_i$-skeleton, namely $W_i$,  it follows that $B_i\cap \Gamma\neq\emptyset$, and Lemma \ref{lem:soul} implies that $\overline{\Gamma\cap B_i}$ is a soul for $B_i$, and $\pi_\Gamma(B_i)=\Gamma\cap B_i$. We let $\Gamma_i:=\overline{\Gamma\cap B_i}$. Note that $\Gamma_i$ contains $\mathrm{diag}(W)\cap B_i$ for $W$ any $n_i$-skeleton intersecting $B_i$. See figure \ref{fig:soul} for an illustration of this construction.

\begin{figure}
\begin{center}
\includegraphics[scale=0.55]{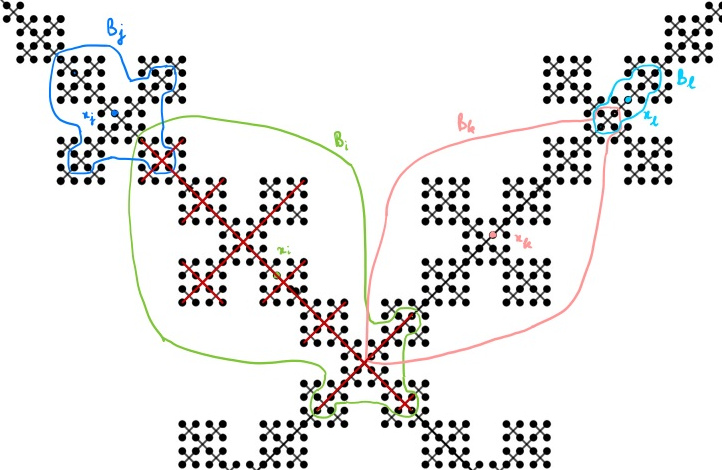}
\end{center}
\caption{The construction of the soul $\Gamma_i$ (in red) in $B_i$ }\label{fig:soul}
\end{figure}

\bigskip

\noindent Actually, we will also need a ``second-order'' variant of this construction, in which the number $n_i$ is replaced by $k_i\leq n_i$ defined as follows: $k_i\in \mathbb{N}^*$ is such that any ball $B_\ell$ having non-empty intersection with one of the balls $B_j$, $j\in J_i$, contains a $k_i$-skeleton (in particular, this must apply to $B_i$ itself). As before, one may take explicitly

$$k_i=\max\left(\lfloor \log_3\left(\frac{r_i}{8c^2}\right)\rfloor,0\right).$$
Following the same steps as above, we get a soul $\tilde{\Gamma}$ for the set $\tilde{S}$, defined as the union of all the $k_i$-skeletons having non-empty intersection with $B_i$. Intersecting with $B_i$, we get a soul $\tilde{\Gamma}_i$ for $B_i$, and $\pi_{\tilde{\Gamma}}(B_i)=\tilde{\Gamma}\cap B_i$. As an illustration, if one takes the example of figure \ref{fig:soul}, then because of the smaller ball $B_\ell$, the soul $\tilde{\Gamma}_i$ of $B_i$ consists of all the edges inside $B_i$. This is indeed a soul which is finer than $\Gamma_i$.

\medskip

The hypotheses of Lemma \ref{lem:Poinc2} are satisfied for $B_i$ and any one of its souls $\Gamma_i$ or $\tilde{\Gamma}_i$, hence the following Poincar\'e inequalities hold true, in which $\Gamma$ denotes either $\Gamma_i$ or $\tilde{\Gamma}_i$: for every $f\in C^\infty(\overline{B_i})$, one has

\begin{equation}\label{eq:P31}
\int_{B_i}|f-c_i(f)|^q\,dm \leq C\,r_i^{\alpha+q-1}\int_{B_i}|\nabla f|^q\,dm
\end{equation}
and

\begin{equation}\label{eq:P32}
\int_{\Gamma}|f-c_i(f)|^q\,dm \leq C\,r_i^q\int_{B_i}|\nabla f|^q\,dm.
\end{equation}
where $c_i(f)$ is a shorthand for $\frac{1}{m(\Gamma)}\int_{\Gamma}f\,dm$, $\Gamma=\Gamma_i$ or $\tilde{\Gamma_i}$.

\medskip

\subsection{A partition of unity associated with a covering by balls} Now we come to the construction of an adapted partition of unity, associated to a covering. Here, we allow the covering to contain small balls. We thus assume that $\Omega$ is an open set which writes

$$\Omega=\bigcup_{i\in I}B_i,$$
where the $B_i$ are balls, and $I\subset \Z$; we assume that for all $i\in I,\,i\geq 0$ (resp. $i<0$), the ball $B_i$ has radius $> 8$ (resp. $\leq 8$), and that moreover the balls $\frac{1}{2}B_i$ still cover $\Omega$. Furthermore, we assume as before that the covering is locally finite, and two intersecting balls have comparable radii. We start with a definition, extending Definition \ref{def:diag}: 

\begin{definition}
{\em 
Let $A\subset X$ and $\Gamma$ be a soul for $A$. A function $f:A\to\R$ is called {\em $\Gamma$-radial}, if for every $p\in A$,

$$f(p)=f\circ\pi_\Gamma(p).$$
}
\end{definition}
For $i\in I$, $i\geq 0$, we denote by $\Gamma_i$ and $\tilde{\Gamma}_i$ the souls of $B_i$ which have been constructed by the previously described procedure. For $i\geq 0$, we first construct functions $\eta_i$ which enjoy the following properties:

\begin{itemize}

\item[(i)] $\eta_i$ is equal to $1$ on $\frac{1}{2}B_i$, and has compact support inside $B_i$,

\item[(ii)] $0\leq \eta_i\leq 1$,

\item[(iii)] $||\nabla \eta_i||_\infty \lesssim \frac{1}{r_i}$,

\item[(iv)] $\eta_i$ is $\Gamma_i$-radial, and as a consequence $\nabla\eta_i$ has support lying inside $\Gamma_i$.

\end{itemize}
For this, fix a smooth function $\varphi : \R_+\to [0,1]$, which is equal to $1$ around $0$, has compact support inside $[0,1]$ and satisfies $0\leq \varphi\leq 1$. Remember that by construction, $\Gamma_i$ is the intersection of $B_i$ with the union of all the diagonals of $n_i$-skeletons (the integer $n_i$ being chosen appropriately, depending on the size of the balls $B_j$ which intersect $B_i$). Consider the union $\Lambda$ of all $ n_i$-skeletons which intersect $\frac{1}{2}B_i$. By definition, if $p\in \Lambda$ is a vertex, then the branch $(\Gamma_i)_p^{B_i}$, lies inside the (unique, if the branch is non-empty) $n_i$-skeleton containing $p$; in particular, it lies in $\Lambda$. Therefore, if $x\in B_i\setminus \Lambda$, then $\pi_{\Gamma_i}(x)\notin \Lambda$. Also, by the choice of $n_i$ made in the construction of $\Gamma_i$
, the diameter of an $n_i$-skeleton is  $2\cdot 3^{n_i}< \frac{r_i}{4}$, so there holds that $\Lambda\subset \frac{3}{4}B_i$. One then lets

$$\eta_i(x)=\begin{cases}
1,&x\in\Lambda\\
\varphi\left(\frac{d(\pi_{\Gamma_i}(x),\Lambda)}{(r_i/8)}\right),& x\in B_i\setminus \Lambda
\end{cases}$$
By definition, $\eta_i$ is a $\Gamma_i$-radial function, so (iv) holds. Note that, for any $x$ lying in one of the $n_i$-skeletons intersecting $B_i$, the distance between $\pi_{\Gamma_i}(x)$ and $x$ is less or equal to half of the diameter of the $n_i$-skeleton containing $x$, which (by the choice of $n_i$) is $<\frac{r_i}{8}$. Since $d(B_i^c,\Lambda)\geq \frac{r_i}{4}$, it follows that

$$d(\pi_{\Gamma_i}(B_i^c),\Lambda)\geq \frac{r_i}{4}-\frac{r_i}{8}=\frac{r_i}{8}.$$
By the fact that $\varphi(t)=0$ for $t\geq 1$, this implies that $\eta_i(x)=0$ for all $x\notin B_i$, hence $\eta_i$ has compact support inside $B_i$. Properties (i) and (ii) of $\eta_i$ now follow right away from the assumptions on $\varphi$. 

For (iii), since $\Lambda\subset \frac{3}{4}B_i$, one has

$$||\nabla \eta_i||_\infty\leq \frac{8}{r_i}||\varphi^\prime||_\infty$$
This concludes the proof of (i), (ii), (iii) and (iv) for $\eta_i$.

We now turn to the definition of $\eta_i$ for $i<0$. In this case, we let 

$$\eta_i(x)=\begin{cases}
1,&x\in\frac{1}{2}B_i,\\
\varphi\left(\frac{2d(x,\frac{1}{2}B_i)}{r_i}\right),& x\in B_i\setminus \frac{1}{2}B_i.
\end{cases}$$
Then, $\eta_i$ satisfy the properties (i), (ii), (iii) above. This concludes the definition of $\eta_i$ for all $i\in I\cap \Z$. Define the function $\eta$ on $\Omega$ by

$$\eta:=\sum_{i\in I}\eta_i$$
(the sum being in fact finite at every point due to the local finiteness of the covering). Since $\eta_i$ is equal to $1$ on $\frac{1}{2}B_i$ and the balls $\frac{1}{2}B_i$ still cover $\Omega$, it follows that $\eta\geq 1$ on $\Omega$. We are now ready to define our partition of unity; for $i\in I$, we let

$$\chi_i:=\frac{\eta_i}{\eta}.$$
Clearly, $\chi_i$ has compact support inside $B_i$, $0\leq \chi_i\leq 1$, and $\sum_{i\in I}\chi_i=1$ in $\Omega$. Moreover, one has

$$\nabla \chi_i=\frac{\nabla \eta_i}{\eta}-\sum_{j\in I}\frac{\eta_i\nabla \eta_j}{\eta^2}$$
(note that in the second sum, there is only a finite number of $j$ such that $\eta_i\nabla \eta_j\neq 0$, namely the $j$'s for which the ball $B_j$ intersect the ball $B_i$). The gradient estimates for $\eta_j$, the fact that intersecting balls of the covering have comparable radii, as well as the fact that $\eta\geq 1$, readily imply that

$$||\nabla \chi_i||_\infty\lesssim\frac{1}{r_i}.$$
Let $J\subset I\cap \N$ denote the set of indices $i$ such that the ball $B_i$ has radius $r_i\geq 9c$ ($c$ being the constant appearing in the hypothesis that two intersecting balls of the covering have comparable radii); then, for every $i\in J$, the ball $B_i$ only intersects balls $B_j$ with radius $r_j>8$\Bk. For $i\in J$, $\nabla\chi_i$ has support inside $\Gamma_i\cup\left(\bigcup_{j\in I} B_i\cap \Gamma_j\right)$. But this set is a subset of $\tilde{\Gamma}_i$. Hence, the following proposition:

\begin{Pro}\label{pro:covering}

There is a partition of unity $(\chi_i)_{i\in I}$ associated with the covering $(B_i)_{i\in I}$ of $\Omega$, with the following properties: 

\begin{itemize}

\item for every $i\in I$, $||\nabla \chi_i||_\infty\lesssim \frac{1}{r_i}$.

\item for every $i\in J$, the support of $\nabla \chi_i$ lies inside $\tilde{\Gamma}_i$.

\end{itemize}

\end{Pro}

\Bk

\subsection{A Calder\'on-Zygmund decomposition}
In this section, we explain how we can get a Calder\'on-Zygmund decomposition in Sobolev spaces, adapted to the Vicsek setting, the proof of which relies on the Poincar\'e inequalities \eqref{eq:P31} and \eqref{eq:P32}.
The statement is as follows:

\begin{lemma}\label{lem:CZ2}
Let $X$ be the Vicsek cable system, and $q\in [1,+\infty)$. Then, there exists a constant $C>0$ depending only on the doubling constant, and $r_0>0$ 
with the following properties: for all $u\in C_0^\infty(X)$ and all $\lambda>0$, there exists a denumerable collection of open balls $(B_i)_{i\in I}$ with radius $r_i$, a denumerable collection of $C^1$ functions $(b_i)_{i\in I}$ and a Lipschitz function $g$ such that:
\begin{enumerate}

\item $\displaystyle  u=g+\sum_{i\in I}b_i$ a.e.,
\item The support of $g$ is included in $\mathrm{supp}(u)$, and $|\nabla g(x)|\lesssim \lambda$, for a.e. $x$.
\item The support of $b_i$ is included in $B_i$, 
$$
 \int_{B_i}|b_i|^q\, dm \lesssim  \max(r_i^{\alpha +q-1},r_i^q)  \int_{B_i}|\nabla u|^q\,dm
$$
and 
$$\displaystyle \int_{B_i} |\nabla b_i|^q\lesssim \int_{B_i}|\nabla u|^q\,dm\lesssim \lambda^q m (B_i).$$

\item $\displaystyle \sum_{i\in I} m (B_i)\lesssim \frac{1}{\lambda^q}\int |\nabla u|^q$.
\item There is a finite upper bound $N$ for the number of balls $B_i$ that have a non-empty intersection.

\item The following inequality holds:

$$||\nabla g||_q\lesssim ||\nabla u||_q.$$
\item There is a constant $c\geq 1$, such that for every $i,j\in I$ with $B_i\cap B_j\neq \emptyset$, the following inequality holds:

$$c^{-1}r_j\leq r_i\leq cr_j.$$

\end{enumerate}
\end{lemma}

\begin{proof}

Let $\lambda>0$. Consider 

$$\Omega=\{x\in X\,;\,\left(M|\nabla u|^q\right)^{1/q}(x)>\lambda\}
$$
where $M$ denotes the uncentered Hardy-Littlewood maximal function, and let $F=X\setminus\Omega$. Let us first prove that \Bk condition \eqref{eq:VE2} implies that $\Omega$ is a bounded open set. For this, let $x_0\in X$ and $R>0$ be such that the support of $u$ is included in $B(x_0,R)$. Let $K>1$ be such that
\begin{equation} \label{eq:conditionK}
\frac 1{\Phi\left(\frac {(K-1)R}2\right)}\int_{X} \left\vert \nabla u(z)\right\vert^q dm(z)\le \lambda.
\end{equation} 
We claim that $\Omega\subset B(x_0,KR)$. Indeed, let $x\in X$ such that $d(x_0,x)\ge KR$ and $B$ be a ball of riadus $r$ containing $x$. If $B\cap B(x_0,R)=\emptyset$, then $\int_B \left\vert \nabla u(z)\right\vert^qdm(z)=0$. Otherwise, 
$$
KR\le d(x_0,x)\le  R+2r,
$$
hence $r\geq \frac{(K-1)R}{2}$, which implies that
$$
\frac 1{m(B)}\int_{B} \left\vert \nabla u(z)\right\vert^q dm(z) \le  \frac 1{\Phi\left(\frac{(K-1)R}2\right)}\int_{X} \left\vert \nabla u(z)\right\vert^q dm(z)\le \lambda.
$$
Therefore, $x\notin \Omega$, which proves the claim, and consequently we have proved that $\Omega$ is indeed a bounded set. \par
\noindent We are going to build a Whitney-type covering of $\Omega$ with some additional properties. For $x\in\Omega$, we let $\bar{B}_x=B(x,\frac{1}{30}d(x,F))$, $\tilde{B}_x:=5\bar{B}_x$ and $B_x:=2\tilde{B}_x=10\bar{B}_x$; note that $B_x\subset \Omega$. We thus get a covering 

$$\Omega= \bigcup_{x\in \Omega}\bar{B}_x.$$
According to the Vitali covering lemma, one can find a disjoint countable subfamily $(\bar{B}_{x_i})_{i\in\N}$ such that

$$\Omega=\bigcup_{i\in\N} 5 \bar{B}_{x_i}=\bigcup_{i\in\N} \tilde{B}_{x_i}.$$
Note that since $B_{x_i}=2\tilde{B}_{x_i}\subset \Omega$ for all $i\in\N$, one has also that

$$\Omega=\bigcup_{i\in\N} B_{x_i}.$$
We now rename $B_i$ the ball $B_{x_i}$. These balls enjoy a certain number of properties that we now list: first, by construction, the balls $(\frac{1}{2}B_i)_{i\geq 0}$, as well as the balls $(B_i)_{i\geq 0}$, form a covering of $\Omega$ by open balls. Also, the fact that the balls $\frac{1}{10}B_i$ are disjoint, together with the fact that the measure is doubling, implies that point (5) of the Calder\'on-Zygmund decomposition is satisfied both for the covering $(B_i)_{i\geq 0}$, and for the covering $(\frac{1}{2}B_i)_{i\geq 0}$. Moreover, for every $i\geq 0$, $10B_i\cap F\neq \emptyset$. Since any point of such a ball $B_k=B(x_k,r_k)$ lies at a distance $\simeq r_k$ from $F$, it follows that if $B_i\cap B_j\neq \emptyset$, then $r_i\simeq r_j$. This is precisely point (7) of the Calder\'on-Zygmund decomposition. The same is true (for the same reasons) for the balls $(\frac{1}{2}B_i)_{i\geq 0}$. These facts will be all the properties of the coverings that will be used in the sequel, and the particular way they have been constructed can now be forgotten by the reader. Relabelling everything, and denoting by $I\subset \Z$ the set of all the new labels, one can assume that the labelling is done in such a way that every ball $B_i$ for $i\in I$, $i\geq 0$ (resp., $i<0$) has radius $>8$ (resp., $\leq 8$), and call $J\subset I\cap\N$ the set of balls $B_i$ which have radius $\geq 9c^2$, so that for $i\in J$, the ball $B_i$ only intersects balls $B_j$ which have radius $>8$, and moreover $B_j$ itself intersects only balls which have themselves radius $>8$.

We now turn to the construction of the functions $b_i$. For that, recall that according to Proposition \ref{pro:covering}, one can find a special partition of unity $(\chi_i)_{i\in I}$ associated with the covering $\Omega=\bigcup_{i\in I}B_i$. For $i\in J$, call $J_i$ the set of $j$'s such that $B_i\cap B_j\neq\emptyset$, and let $K_i$ the set of $k$'s such that there exists $j\in J_i$ for which $B_j\cap B_k\neq\emptyset$. Recall that for every $i\in J$, we have constructed souls $\tilde{\Gamma}_i$ of $B_i$, which are adapted to the covering $B_i\cup \left(\cup_{k\in K_i} B_k\right)\cup\left(\cup_{j\in J_i}B_j\right)$ of $B_i$ (note that by definition of $J$, every ball in this covering has radius $>8$). We now let

$$b_i=(u-c_i)\chi_i,,$$
where, for $i\in I\setminus J$,

$$c_i=\frac{1}{m(B_i)}\int_{B_i}u\,dm,$$
while for $i\in J$,

$$c_i=\frac{1}{m(\tilde{\Gamma}_i)}\int_{\tilde{\Gamma}_i}u\,dm.$$
According to the Poincar\'e inequality for balls with small ($\leq 9c^2$) radii on the Vicsek cable system and \eqref{eq:P31} for $\Gamma=\tilde{\Gamma}_i$, one has for every $i\in I$,

$$\int_{B_i}|b_i|^q\, dm \lesssim  \max(r_i^{\alpha +q-1},r_i^q)  \int_{B_i}|\nabla u|^q\,dm.$$
Since 

$$\nabla b_i=(\nabla u)\cdot \chi_i+(u-c_i)(\nabla \chi_i),$$
using on the one hand for $i\in I\setminus J$ Poincar\'e inequality for balls with small radii on the Vicsek cable system , and on the other hand for $i\in J$ inequality \eqref{eq:P32} for $\Gamma=\tilde{\Gamma}_i$ together with the fact that $\nabla \chi_i$ has support on $\tilde{\Gamma}_i$, one concludes that

$$\int_{B_i}|\nabla b_i|^q\,dm \lesssim \int_{B_i}|\nabla u|^q \,dm.$$
\Bk Since $\int_{10 B_i}|\nabla u|^q\,dm\leq \lambda^q V(10B_i)$ because $10B_i \cap F\neq \emptyset$ and by definition of the maximal function, one deduces from doubling that

$$\int_{B_i}|\nabla u|^q\,dm \lesssim \lambda^q V(B_i),$$
which implies property (3). The property (4) follows from (5) and the fact that according to the weak (1,1) type of $M$,

$$m(\Omega)\lesssim \frac 1{\lambda^q} \int_X |\nabla u|^q\,dm.$$
One shows as in \cite[Proposition 1.1]{AC} that the series $\sum_{i\in I} b_i$ converges in $L^1_{loc}$. We define

$$g:=u-\sum_{i\in I}b_i,$$
which is thus a well-defined function in $L^1_{loc}$. From this point on, we complete the proof following the arguments in \cite[Section 4]{DRGauss}.


\end{proof}

\subsection{Reverse quasi-Riesz inequalities: the positive result}

We are now ready for the proof of Theorem \ref{thm:RRp-subg} in the positive case $p>p^*$, following the approach laid out in \cite{AC}. Thus, we assume that $\gamma\in [\frac{1}{2},1)$. Recall that $p^*$ is defined by

$$p^*=\begin{cases}
\frac{\alpha-1}{\gamma(\alpha+1)-1}& \mbox{ if }\gamma\in (\frac{1}{\alpha+1},\frac{\alpha}{\alpha+1}),\\
+\infty& \mbox{ if }\gamma\leq \frac{1}{\alpha+1},\\
1& \mbox{ if }\gamma\geq \frac{\alpha}{\alpha+1}.
\end{cases}$$
We use the following resolution of $e^{-\Delta}\Delta^{\gamma}$:

$$e^{-\Delta}\Delta^{\gamma}=c\int_0^\infty \frac{\partial}{\partial t}e^{-(t+1)\Delta}\,\frac{dt}{t^{\gamma}}.$$
In what follows, we will not write down the constant $c$ anymore. According to Lemma \ref{lem:small_time}, if we denote

$$T=\int_0^{1} \frac{\partial}{\partial t}e^{-(t+1)\Delta}\,\frac{dt}{t^{\gamma}},$$
then for every $p\in (1,\infty)$,

$$||Tf||_p\lesssim ||\nabla f||_p,\quad f\in C_0^\infty(X).$$
Hence, letting

$$\mathscr{R}_\gamma=\int_{1}^\infty \frac{\partial}{\partial t}e^{-(t+1)\Delta}\,\frac{dt}{t^{\gamma}},$$
in order to prove Theorem \ref{thm:RRp-subg} it is enough to show that

\begin{equation}
||\mathscr{R}_\gamma f||_p\lesssim ||\nabla f||_p,\quad f\in C_0^\infty(X).
\end{equation}
We first prove the following weak type $(q,q)$ estimate:

\begin{Lem}\label{lem:weak}

Let $q\in [1,2)$, and assume that $q\geq p^*$. Then, there exists a constant $C>0$ such that, for every $f\in C_0^\infty(X)$ and every $\lambda>0$, one has

\begin{equation}\label{eq:weak}
m\left(\{|\mathscr{R}_\gamma f|>\lambda\}\right) \leq \frac{C}{\lambda^q}\int_X|\nabla f|^q\,dm.
\end{equation}

\end{Lem}

\begin{proof}
Recall the Calder\'on-Zygmund decomposition in Sobolev spaces from Lemma \ref{lem:CZ2}, and decompose $f=g+\sum_{i\in I} b_i$ accordingly.
One has

\begin{eqnarray*}
m\left(\{|\mathscr{R}_\gamma f|> 2\lambda\}\right) & \leq  & m\left(\{|\mathscr{R}_\gamma  g|>\lambda\}\right)+m\left(\left\{\left|\mathscr{R}_\gamma \left( \sum_{i\in I}b_i\right)\right|>\lambda\right\}\right)\\
& =: & A+B.
\end{eqnarray*}
We first treat the term $A$: we have by the Chebyshev inequality

$$m\left(\{|\mathscr{R}_\gamma g|>\lambda\}\right)=m\left(\{|\mathscr{R}_\gamma g|^2>\lambda^{2}\}\right) \leq \frac{1}{\lambda^{2}} \int_X|\mathscr{R}_\gamma g|^2 \,dm.$$
According to Lemma \ref{lem:RR2} (since $\varepsilon=1-\gamma\leq \frac{1}{2}$) and Lemma \ref{lem:small_time}, one has

$$||\mathscr{R}_\gamma g||_2\lesssim ||\nabla g||_2.$$
But as a consequence of the Calder\'on-Zygmund decomposition, $||\nabla g||_\infty\lesssim \lambda$ and $\int_X|\nabla g|^q\,dm \lesssim \int_X|\nabla f|^q\,dm$, hence

$$m\left(\{|\mathscr{R}_\gamma g|^2>\lambda^2\}\right)\lesssim \frac{1}{\lambda^q}\int_X|\nabla f|^q\,dm.$$
We now turn to estimate the term $B=m\left(\left\{\left|\mathscr{R}_\gamma \left( \sum_{i\in I}b_i\right)\right|>\lambda\right\}\right)$. Let $J\subset I$ denotes the set of indices $j\in I$ for which the ball $B_j$ has radius $r_j> 1$. For $i\in I\setminus J$, write

\bean
\mathscr{R}_\gamma b_i &= &\int_{1}^\infty \frac{\partial}{\partial t}e^{-(t+1)\Delta}b_i\,\frac{dt}{t^{\gamma}}\\
& =: & U_0b_i,
\eean
while for $i\in J$, write

\bean
\mathscr{R}_\gamma b_i &= & \int_{1}^{r_i^\beta} \frac{\partial}{\partial t}e^{-(t+1)\Delta}b_i\,\frac{dt}{t^{\gamma}}+\int_{r_i^\beta}^\infty \frac{\partial}{\partial t}e^{-(t+1)\Delta}b_i\,\frac{dt}{t^{\gamma}}\\
& = & T_ib_i+U_ib_i.
\eean
The two lemmas below (Lemmas \ref{lem:Ta2} and \ref{lem:Ua}) imply that

$$\sum_{i\in I}||\mathscr{R}_\gamma b_i||_{L^q(X\setminus 4B_i)}^q\lesssim \sum_{i\in I\setminus J}||b_i||_q^q+\sum_{i\in J}\frac{1}{r_i^{q\beta\gamma}}||b_i||_q^q.$$
The sum $\sum_{i\in I\setminus J}||b_i||_q^q$ is easily estimated, thanks to the properties of the functions $b_i$: indeed, one has for $i\in I\setminus J$,

$$||b_i||_q\lesssim r_i ||\nabla f||_{L^q(B_i)}\leq r_0 ||\nabla f||_{L^q(B_i)}.$$
Since the balls $B_i$ have the finite intersection property,

$$\sum_{i\in I\setminus J}||b_i||_q^q\lesssim ||\nabla f||_q^q.$$
Finally,

\bean
m\left(\left\{\left|\mathscr{R}_\gamma \left( \sum_{i\in I}b_i\right)\right|>\lambda\right\}\right)  & \leq & m\left((X\setminus \bigcup_{i\in I}4B_i)\cap \left\{\left|\mathscr{R}_\gamma \left( \sum_{i\in I}b_i\right)\right|>\lambda\right\}\right)+m\left(\bigcup_{i\in I}4B_i\right)\\
&\lesssim & m\left(\left(X\setminus \bigcup_{i\in I}4B_i\right)\cap \left\{\left|\mathscr{R}_\gamma \left( \sum_{i\in I}b_i\right)\right|>\lambda\right\}\right)+\sum_{i\in I}m(B_i)\\
&\lesssim & m\left(\left(X\setminus \bigcup_{i\in I}4B_i\right)\cap \left\{\left|\mathscr{R}_\gamma \left( \sum_{i\in I}b_i\right)\right|>\lambda\right\}\right)+\lambda^{-q}\int_X|\nabla f|^q\,dm\\
&\lesssim & \lambda^{-q}\sum_{i\in I}||\mathscr{R}_\gamma b_i||^q_{L^q(X\setminus 4B_i)}+\lambda^{-q}\int_X|\nabla f|^q\,dm\\
&\lesssim & \frac{1}{\lambda^q}\left(\int_X|\nabla f|^q\,dm+\sum_{i\in J} \frac{1}{r_i^{q\beta\gamma }}\int_{B_i}|b_i|^q\,dm\right).
\eean
Property $(4)$ in Lemma \ref{lem:CZ2} yields, for all $i\in J$, 
$$\int_{B_i}|b_i|^q\,dm\lesssim r_i^{\alpha +q-1}\int_{B_i}|\nabla f|^q\,dm.$$
Elementary computations, using the fact that $\beta=\alpha+1$, show that the condition $q\geq p^*$ is equivalent to 

$$\alpha+q-1\leq q\beta\gamma.$$
Hence, we get for all $i\in J$,

$$\frac{1}{r_i^{q\beta\gamma}}\int_{B_i}|b_i|^q\,dm\lesssim \int_{B_i}|\nabla f|^q\,dm,$$
and given the finite intersection property of the balls $B_i$, we arrive to

$$m\left(\left\{\left|\mathscr{R}_\gamma \left( \sum_{i\in I}b_i\right)\right|>\lambda\right\}\right)\lesssim \frac{1}{\lambda^q} \int_X|\nabla f|^q\,dm.$$
This completes the proof of Lemma \ref{lem:weak}.


\end{proof}
The following two lemmas have been used in the proof of Lemma \ref{lem:weak}:

\begin{Lem}\label{lem:Ta2}

For every $i\in J$,

$$|| T_i b_i||_{L^q(X\setminus 4B_i)}\lesssim \frac{1}{r_i^{\beta \gamma}}\left\Vert b_i\right\Vert_{q}.$$

\end{Lem}

\begin{Lem}\label{lem:Ua}

For every $i\in I\setminus J$,

$$||U_0 b_i||_q\lesssim ||b_i||_q,$$
and for every $i\in J$,

$$||U_ib_i||_q\lesssim \frac{1}{r_i^{\beta \gamma}}||b_i||_q.$$

\end{Lem}
Before we prove these two results, recall that the subgaussian estimate for the heat kernel on $X$ implies, by \cite[Theorem 4]{Da}, the following pointwise bounds for the time derivative of the heat kernel:
\begin{equation}\label{eq:dt_pt}
\left\vert \frac{\partial }{\partial t}p_t(x,y)\right\vert\le
\begin{cases}
\frac{C_1}{tV(x,\sqrt{t})}\exp\left(-C_2\frac{d(x,y)^2}{t}\right),&\text{if }t\in(0,1),\\
\frac{C_1}{tV(x,t^{1/\beta})}\exp\left(-C_2\left(\frac{d(x,y)}{t^{{1}/{\beta}}}\right)^{\frac{\beta}{\beta-1}}\right),&\text{if }t\in[1,+\infty).
\end{cases}
\end{equation}

\begin{proof}[Proof of Lemma \ref{lem:Ta2}]

Let $i\in J
$, and $t\in (1,r_i^\beta)$. Note that $X\setminus 4B_i=\bigcup_{j\geq 2}C_i^j$ where, for all $j\ge 2$, $C_i^j:=2^{j+1}B_i\setminus 2^jB_i$. Denote also by $x_i$ the center of $B_i$. We start by estimating $\left\vert\frac{\partial}{\partial t}e^{-(t+1)\Delta}b_i\right\vert$ pointwise on $C_i^j$, $j\geq 2$. So, let $j\geq 2$. Notice that \eqref{eq:VE2} implies that for every $z\in B_i$,

\begin{eqnarray*}
\frac{V(x_i,(t+1)^{1/\beta})}{V(z,(t+1)^{1/\beta})} &=& \frac{V(x_i,(t+1)^{1/\beta})}{V(x_i,r_i)}\cdot \frac{V(x_i,r_i)}{V(z,r_i)}\cdot \frac{V(z,r_i)}{V(z,(t+1)^{1/\beta})}\\
&\lesssim & \left(\frac{r_i}{(t+1)^{1/\beta}}\right)^\alpha
\end{eqnarray*}
(where we have used that $\frac{V(x_i,(t+1)^\beta)}{V(x_i,r_i)}\lesssim 1$ because $(t+1)^\beta\lesssim r_i$). Bearing in mind that $b_i$ has support in $B_i$ and using \eqref{eq:dt_pt}, one obtains, for all $x\in C_i^j$, 

\begin{eqnarray*}
\left| \frac{\partial}{\partial t}e^{-(t+1)\Delta}b_i(x)\right| & \lesssim & \frac{1}{t+1}\Bk \left(\frac{r_i}{(t+1)^{1/\beta}}\right)^\alpha \Bk e^{-c\left(\frac{2^jr_i}{(t+1)^{1/\beta}}\right)^{\frac{\beta}{\beta-1}}} \frac{V(x_i,r_i)}{V(x_i,(t+1)^{1/\beta})}\fint_{B_i}|b_i(z)|\,dm(z)\\
&\lesssim & \frac{1}{t+1}\left(\frac{r_i}{(t+1)^{1/\beta}}\right)^{2\alpha} e^{-c\left(\frac{2^jr_i}{(t+1)^{1/\beta}}\right)^{\frac{\beta}{\beta-1}}} \left(\fint_{B_i}|b_i(z)|^q\,dm(z)\right)^{1/q}\\
& = & \frac 1{r_{i}^\beta} \left(\frac{r_i}{(t+1)^{1/\beta}}\right)^{2\alpha+\beta} e^{-c\left(\frac{2^jr_i}{(t+1)^{1/\beta}}\right)^{\frac{\beta}{\beta-1}}} \left(\fint_{B_i}|b_i(z)|^q\,dm(z)\right)^{1/q},\\
\end{eqnarray*}
where in the second line we have used \eqref{eq:VE2} and H\"older's inequality. As a consequence,

\begin{eqnarray*}
\left\Vert \frac{\partial}{\partial_t} e^{-(t+1)\Delta}b_{i}\right\Vert_{L^q(C_i^j)} & \le & m(C_i^j)^{1/q} \left\Vert \frac{\partial}{\partial_t} e^{-(t+1)\Delta}b_{i}\right\Vert_{L^{\infty}(C_i^j)}\\
& \lesssim & m(2^{j+1}B_{i})^{1/q}\frac 1{r_{i}^\beta} \left(\frac{r_i}{(t+1)^{1/\beta}}\right)^{2\alpha+\beta} e^{-c\left(\frac{2^jr_i}{(t+1)^{1/\beta}}\right)^{\frac{\beta}{\beta-1}}} \left(\fint_{B_i}|b_i(z)|^q\,dm(z)\right)^{1/q}\\
& \lesssim &  2^{j\alpha /q}\frac 1{r_{i}^\beta} \left(\frac{r_i}{(t+1)^{1/\beta}}\right)^{2\alpha+\beta} e^{-c\left(\frac{2^jr_i}{(t+1)^{1/\beta}}\right)^{\frac{\beta}{\beta-1}}} \left(\int_{B_i}|b_i(z)|^q\,dm(z)\right)^{1/q}\\
&\lesssim &  2^{j\alpha /q}\frac 1{r_{i}^{\beta(1-\gamma)}} \left(\frac{r_i}{(t+1)^{1/\beta}}\right)^{2\alpha+\beta} e^{-c\left(\frac{2^jr_i}{(t+1)^{1/\beta}}\right)^{\frac{\beta}{\beta-1}}} \left(\int_{B_i}\left(\frac{|b_i(z)|}{r_i^{\beta\gamma}}\right)^q\,dm(z)\right)^{1/q}.
\end{eqnarray*}
Now we estimate

$$I=\int_0^{r_i^\beta}\left(\frac{r_i}{(t+1)^{1/\beta}}\right)^{2\alpha+\beta} e^{-c\left(\frac{2^jr_i}{(t+1)^{1/\beta}}\right)^{\frac{\beta}{\beta-1}}} \frac{dt}{t^{\gamma}}.$$
We make the change of variable $s=\frac{2^jr_i}{t^{1/\beta}}$, so $\frac{ds}{s}=-\frac{1}{\beta}\frac{dt}{t}$ hence

$$I=\beta(2^jr_i)^{\beta(1-\gamma)}\int_{2^j}^\infty \left(\frac{r_{i}}{\left(\left(\frac{2^jr_{i}}s\right)^{\beta}+1\right)^{1/\beta}}\right)^{2\alpha+\beta}\exp\left(-c\left(\frac{2^jr_i}{\left(\left(\frac{2^jr_{i}}s\right)^{\beta}+1\right)^{1/\beta}}\right)^{\beta/\beta-1}\right)\,\frac{ds}{s^{1+\beta(1-\gamma)}}.$$
The elementary inequality

$$x^{2\alpha+\beta}e^{-c'(2^jx)^{\beta/\beta-1}}\leq C, \quad \forall j\geq 2,\,\forall x\geq 0,$$
with $c'=\frac{c}{2}$, and \Bk $x=\frac{r_{i}}{\left(\left(\frac{2^jr_{i}}s\right)^{\beta}+1\right)^{1/\beta}}$ \Bk entails that

$$I\lesssim (2^jr_i)^{\beta(1-\gamma)} \Bk\int_{2^j}^\infty \exp\left(-c'\left(\frac{2^jr_i}{\left(\left(\frac{2^jr_{i}}s\right)^{\beta}+1\right)^{1/\beta}}\right)^{\beta/\beta-1}\right)\,\frac{ds}{s^{1+\beta(1-\gamma)}}.$$
Given our assumption that $r_i\geq r_0$, observe that for $s\geq 2^j$,
$$
\left(\frac{2^jr_{i}}s\right)^{\beta}+1\le r_{i}^{\beta}+1\lesssim r_{i}^{\beta},
$$
so that one can estimate the exponential factor in the integrand by

$$e^{-c'2^{\frac{j\beta}{\beta-1}}},$$
therefore

\begin{eqnarray*}
I&\lesssim & r_i^{\beta(1-\gamma)} e^{-c'2^{\frac{j\beta}{\beta-1}}}\int_{2^j}^\infty \left(\frac{2^j}{s}\right)^{\beta(1-\gamma)}\,\frac{ds}{s}\\
&\lesssim &   r_i^{\beta(1-\gamma)} e^{-c'2^{\frac{j\beta}{\beta-1}}} \int_1^\infty \frac{du}{u^{1+\beta(1-\gamma)}}\\
&\lesssim & r_i^{\beta(1-\gamma)} e^{-c'2^{\frac{j\beta}{\beta-1}}}.
\end{eqnarray*}
Finally, one obtains that

$$\left\Vert  T_i b_i\right\Vert_{L^q(C_i^j)}\lesssim  2^{j\alpha/q} e^{-c'2^{\frac{j\beta}{\beta-1}}} \left(\int_{B_i}\left(\frac{|b_i(z)|}{r_i^{\beta\gamma}}\right)^q\,dm(z)\right)^{1/q},$$
hence summing over $j\geq 2$,

$$\left\Vert  T_i b_i\right\Vert_{L^q(X\setminus 4B_i)}\lesssim \left(\int_{B_i}\left(\frac{|b_i(z)|}{r_i^{\beta\gamma}}\right)^q\,dm(z)\right)^{1/q}.$$

\end{proof}

\begin{proof}[Proof of Lemma \ref{lem:Ua}]

We first assume that $i\in I\setminus J$. We write

$$U_0 b_i=\int_{1}^\infty t\Delta e^{-t\Delta}\left(\frac{b_i}{t^{\gamma}}\right)\,\frac{dt}{t}= \int_0^\infty t\Delta e^{-t\Delta}b_t\,\frac{dt}{t},$$
with

$$b_t=\frac{b_i}{t^{\gamma}}\mathbf{1}_{t\geq 1}.$$
Let $g\in L^{q^{\prime}}$ with $\frac 1q+\frac 1{q^{\prime}}=1$, then since $ q^{\prime}\in (1,+\infty)$, Littlewood-Paley-Stein estimates (\cite[Chapter 4, Theorem 10]{topics}) yield

\begin{eqnarray*}
\left\vert \int_X (U_0 b_i )g\right\vert  & = &\left\vert  \int_0^\infty \langle t\Delta e^{-t\Delta} b_t,g\rangle \frac{dt}{t}\right\vert  \\
&= & \left\vert  \int_0^\infty \langle b_t,  t\Delta e^{-t\Delta} g\rangle \frac{dt}{t}\right\vert \\
&\leq & \left\Vert \left(\int_0^\infty |b_t|^2\frac{dt}{t}\right)^{1/2}\right \Vert_{q} \left\Vert \left(\int_0^\infty |t\Delta e^{-t\Delta}g|^2\frac{dt}{t}\right)^{1/2}\right \Vert_{q^{\prime}}\\
&\lesssim &\left \Vert \left(\int_0^\infty |b_t|^2\frac{dt}{t}\right)^{1/2}\right \Vert_{q} ||g||_{q^{\prime}}.
\end{eqnarray*}
It is easily seen that

$$\left \Vert \left(\int_0^\infty |b_t|^2\frac{dt}{t}\right)^{1/2}\right \Vert_{q}\lesssim  ||b_i||_q,$$
hence

$$\left\vert \int_X (U_0 b_i ) g\right\vert \lesssim ||b_i||_q||g||_{q^{\prime}}.$$
Taking the sup over all $g\in L^{q^{\prime}} (X)$ with $\left\Vert g\right\Vert_{q^{\prime}}\le 1$, we get

$$||U_0 b_i||_{ q}\lesssim ||b_i||_q.$$
The proof for $i\in J$ is similar, in this case one has 

$$b_t=\frac{b_i}{t^{\gamma}}\mathbf{1}_{t\ge r_i^\beta},$$
which leads to the estimate

$$\left \Vert \left(\int_0^\infty |b_t|^2\frac{dt}{t}\right)^{1/2}\right \Vert_{q}\lesssim \frac{1}{r_i^{\beta\gamma
}}||b_i||_q.$$
The rest of the argument is identical to the case $ i\in I\setminus J$.

\end{proof}

\begin{proof}[End of the proof of Theorem \ref{thm:RRp-subg}, in the case $p>p^*$:] we conclude by an interpolation argument similar to the one from \cite[Lemma 3.4]{DRGauss}, relying on Lemma \ref{lem:weak} applied with $q=p^{\ast}$.

\end{proof}
\subsection{Reverse quasi-Riesz inequalities: negative results}
\begin{proof}[Proof of Theorem \ref{thm:RRp-subg}, in the case $p<p^*$]
We first assume that $\gamma>\frac{1}{\beta}$. As in \cite[Section 5]{CCFR}, we start from the Nash inequality
\begin{equation} \label{Nash}
\left\Vert f\right\Vert_p^{1+\frac{2\gamma p}{(p-1)\alpha^{\prime}}}\lesssim \left\Vert f\right\Vert_1^{\frac{2\gamma p}{(p-1)\alpha^{\prime}}}\left\Vert \Delta^{\gamma}f\right\Vert_p
\end{equation}
with $\alpha^{\prime}:=\frac{2\alpha}{\alpha+1}$, whenever $\left\Vert f\right\Vert_p\le \left\Vert f\right\Vert_1$. The Nash inequality \eqref{Nash} holds for any $\gamma>0$, thanks to the sub-Gaussian upper-bound of the heat kernel for large times (see \cite[Theorem 1]{Cou}). With the notations of \cite{DRY}, we write $X$ as the increasing union of $V^{(n)}$, $n\in\N$. Let $U_0=U_1=U_2=U_3=V^{(1)}$. By definition, for all $k\in\N$, each $V^{(k+1)}$ is the union of $V^{(k)}$ and of 4 additional translated copies of $V^{(k)}$. Call $U_{4k+i}$, $i=0,\cdots,3$ the collection of these four copies, enumerated such that $U_{4k}$ is the central copy and $U_{4k+3}$ the upper right copy. See figure \ref{fig_U4k}.\par

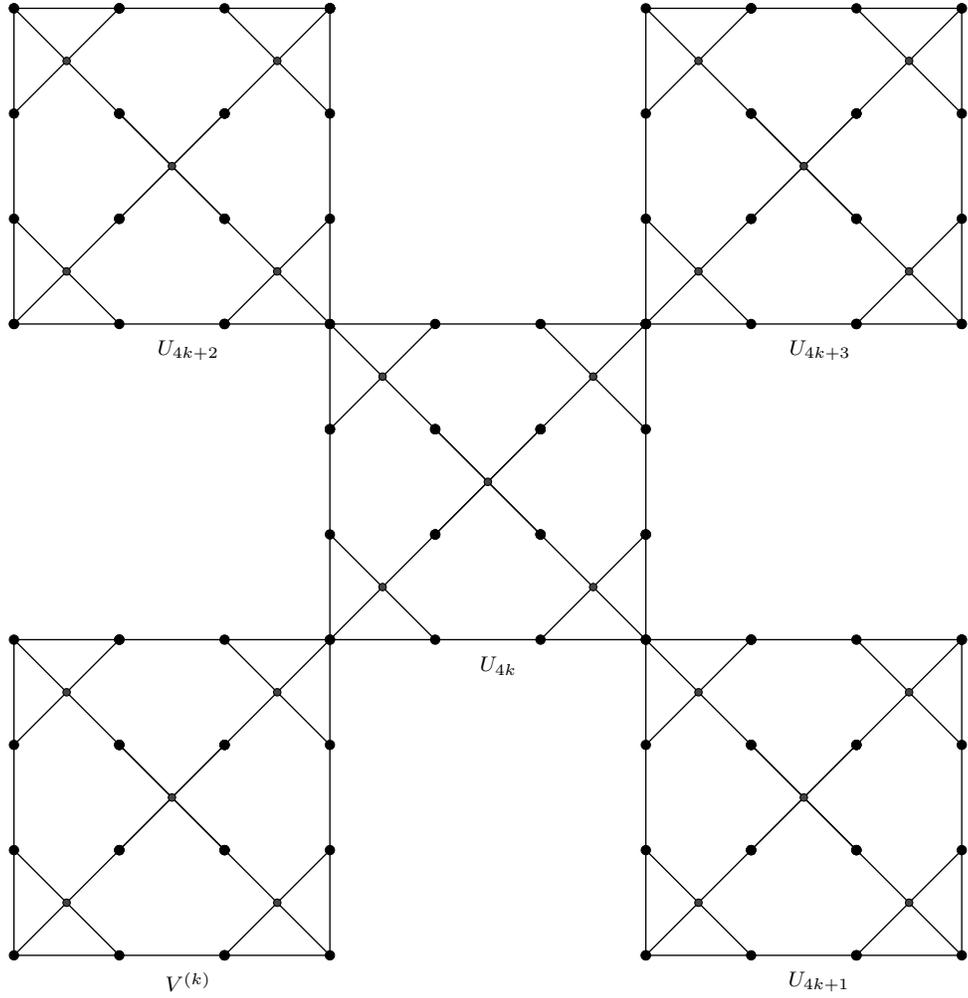
\begin{figure}[ht]
\centering
\begin{tikzpicture}[scale=0.7][line cap=round,line join=round,>=triangle 45,x=1cm,y=1cm]
\clip(-2,-3.4) rectangle (29.74,19.42);
\draw [line width=0.5pt] (-1,-1)-- (1,1);
\draw [line width=0.5pt] (-1,1)-- (1,-1);
\draw [line width=0.5pt] (3,3)-- (1,1);
\draw [line width=0.5pt] (1,3)-- (3,1);
\draw [line width=0.5pt] (5,1)-- (3,-1);
\draw [line width=0.5pt] (5,-1)-- (3,1);
\draw [line width=0.5pt] (5,5)-- (3,3);
\draw [line width=0.5pt] (5,3)-- (3,5);
\draw [line width=0.5pt] (1,1)-- (3,3);
\draw [line width=0.5pt] (3,1)-- (1,3);
\draw [line width=0.5pt] (-1,3)-- (1,5);
\draw [line width=0.5pt] (-1,5)-- (1,3);
\draw [line width=0.5pt] (11,11)-- (9,9);
\draw [line width=0.5pt] (9,11)-- (11,9);
\draw [line width=0.5pt] (7,7)-- (9,9);
\draw [line width=0.5pt] (7,9)-- (9,7);
\draw [line width=0.5pt] (9,5)-- (11,7);
\draw [line width=0.5pt] (11,5)-- (9,7);
\draw [line width=0.5pt] (5,5)-- (7,7);
\draw [line width=0.5pt] (7,5)-- (5,7);
\draw [line width=0.5pt] (9,9)-- (7,7);
\draw [line width=0.5pt] (9,7)-- (7,9);
\draw [line width=0.5pt] (7,11)-- (5,9);
\draw [line width=0.5pt] (5,11)-- (7,9);
\draw [line width=0.5pt] (17,5)-- (15,3);
\draw [line width=0.5pt] (17,3)-- (15,5);
\draw [line width=0.5pt] (13,1)-- (15,3);
\draw [line width=0.5pt] (15,1)-- (13,3);
\draw [line width=0.5pt] (11,3)-- (13,5);
\draw [line width=0.5pt] (11,5)-- (13,3);
\draw [line width=0.5pt] (11,-1)-- (13,1);
\draw [line width=0.5pt] (11,1)-- (13,-1);
\draw [line width=0.5pt] (15,3)-- (13,1);
\draw [line width=0.5pt] (13,3)-- (15,1);
\draw [line width=0.5pt] (17,1)-- (15,-1);
\draw [line width=0.5pt] (17,-1)-- (15,1);
\draw [line width=0.5pt] (11,11)-- (13,13);
\draw [line width=0.5pt] (11,13)-- (13,11);
\draw [line width=0.5pt] (15,15)-- (13,13);
\draw [line width=0.5pt] (13,15)-- (15,13);
\draw [line width=0.5pt] (17,13)-- (15,11);
\draw [line width=0.5pt] (17,11)-- (15,13);
\draw [line width=0.5pt] (17,17)-- (15,15);
\draw [line width=0.5pt] (17,15)-- (15,17);
\draw [line width=0.5pt] (13,13)-- (15,15);
\draw [line width=0.5pt] (15,13)-- (13,15);
\draw [line width=0.5pt] (11,15)-- (13,17);
\draw [line width=0.5pt] (11,17)-- (13,15);
\draw [line width=0.5pt] (5,17)-- (3,15);
\draw [line width=0.5pt] (5,15)-- (3,17);
\draw [line width=0.5pt] (1,13)-- (3,15);
\draw [line width=0.5pt] (3,13)-- (1,15);
\draw [line width=0.5pt] (-1,15)-- (1,17);
\draw [line width=0.5pt] (-1,17)-- (1,15);
\draw [line width=0.5pt] (-1,11)-- (1,13);
\draw [line width=0.5pt] (-1,13)-- (1,11);
\draw [line width=0.5pt] (3,15)-- (1,13);
\draw [line width=0.5pt] (1,15)-- (3,13);
\draw [line width=0.5pt] (5,13)-- (3,11);
\draw [line width=0.5pt] (5,11)-- (3,13);
\draw [line width=0.5pt] (-1,5)-- (5,5);
\draw [line width=0.5pt] (5,5)-- (5,-1);
\draw [line width=0.5pt] (5,-1)-- (-1,-1);
\draw [line width=0.5pt] (-1,5)-- (-1,-1);
\draw [line width=0.5pt] (11,5)-- (11,-1);
\draw [line width=0.5pt] (11,-1)-- (17,-1);
\draw [line width=0.5pt] (11,5)-- (17,5);
\draw [line width=0.5pt] (17,5)-- (17,-1);
\draw [line width=0.5pt] (5,5)-- (11,5);
\draw [line width=0.5pt] (11,11)-- (11,5);
\draw [line width=0.5pt] (5,11)-- (5,5);
\draw [line width=0.5pt] (5,11)-- (11,11);
\draw [line width=0.5pt] (11,11)-- (11,17);
\draw [line width=0.5pt] (11,17)-- (17,17);
\draw [line width=0.5pt] (17,17)-- (17,11);
\draw [line width=0.5pt] (17,11)-- (11,11);
\draw [line width=0.5pt] (5,11)-- (-1,11);
\draw [line width=0.5pt] (-1,11)-- (-1,17);
\draw [line width=0.5pt] (-1,17)-- (5,17);
\draw [line width=0.5pt] (5,17)-- (5,11);
\begin{scriptsize}
\draw [fill=uuuuuu] (0,0) circle (2pt);
\draw [fill=black] (1,1) circle (2.5pt);
\draw [fill=black] (1,-1) circle (2.5pt);
\draw [fill=black] (-1,-1) circle (2.5pt);
\draw [fill=black] (-1,1) circle (2.5pt);
\draw [fill=black] (3,-1) circle (2.5pt);
\draw [fill=black] (3,3) circle (2.5pt);
\draw [fill=black] (1,1) circle (2.5pt);
\draw [fill=black] (1,3) circle (2.5pt);
\draw [fill=black] (3,1) circle (2.5pt);
\draw [fill=uuuuuu] (2,2) circle (2pt);
\draw [fill=black] (5,1) circle (2.5pt);
\draw [fill=black] (3,-1) circle (2.5pt);
\draw [fill=black] (5,-1) circle (2.5pt);
\draw [fill=black] (3,1) circle (2.5pt);
\draw [fill=black] (3,-1) circle (2.5pt);
\draw [fill=uuuuuu] (4,0) circle (2pt);
\draw [fill=black] (5,5) circle (2.5pt);
\draw [fill=black] (3,3) circle (2.5pt);
\draw [fill=black] (5,3) circle (2.5pt);
\draw [fill=black] (3,5) circle (2.5pt);
\draw [fill=black] (1,1) circle (2.5pt);
\draw [fill=black] (3,3) circle (2.5pt);
\draw [fill=black] (3,1) circle (2.5pt);
\draw [fill=black] (1,3) circle (2.5pt);
\draw [fill=black] (-1,3) circle (2.5pt);
\draw [fill=black] (1,5) circle (2.5pt);
\draw [fill=black] (-1,5) circle (2.5pt);
\draw [fill=black] (1,3) circle (2.5pt);
\draw [fill=uuuuuu] (4,4) circle (2pt);
\draw [fill=black] (1,5) circle (2.5pt);
\draw [fill=black] (1,5) circle (2.5pt);
\draw [fill=uuuuuu] (0,4) circle (2pt);
\draw [fill=black] (11,11) circle (2.5pt);
\draw [fill=black] (9,9) circle (2.5pt);
\draw [fill=black] (9,11) circle (2.5pt);
\draw [fill=black] (11,9) circle (2.5pt);
\draw [fill=black] (7,7) circle (2.5pt);
\draw [fill=black] (9,9) circle (2.5pt);
\draw [fill=black] (7,9) circle (2.5pt);
\draw [fill=black] (9,7) circle (2.5pt);
\draw [fill=black] (9,5) circle (2.5pt);
\draw [fill=black] (11,7) circle (2.5pt);
\draw [fill=black] (11,5) circle (2.5pt);
\draw [fill=black] (9,7) circle (2.5pt);
\draw [fill=black] (5,5) circle (2.5pt);
\draw [fill=black] (7,7) circle (2.5pt);
\draw [fill=black] (7,5) circle (2.5pt);
\draw [fill=black] (5,7) circle (2.5pt);
\draw [fill=black] (9,9) circle (2.5pt);
\draw [fill=black] (7,7) circle (2.5pt);
\draw [fill=black] (9,7) circle (2.5pt);
\draw [fill=black] (7,9) circle (2.5pt);
\draw [fill=black] (7,11) circle (2.5pt);
\draw [fill=black] (5,9) circle (2.5pt);
\draw [fill=black] (5,11) circle (2.5pt);
\draw [fill=black] (7,9) circle (2.5pt);
\draw [fill=uuuuuu] (10,10) circle (2pt);
\draw [fill=black] (11,7) circle (2.5pt);
\draw [fill=uuuuuu] (8,8) circle (2pt);
\draw [fill=black] (11,7) circle (2.5pt);
\draw [fill=uuuuuu] (10,6) circle (2pt);
\draw [fill=uuuuuu] (6,6) circle (2pt);
\draw [fill=black] (5,9) circle (2.5pt);
\draw [fill=black] (5,9) circle (2.5pt);
\draw [fill=uuuuuu] (6,10) circle (2pt);
\draw [fill=black] (11,11) circle (2.5pt);
\draw [fill=black] (11,11) circle (2.5pt);
\draw [fill=black] (17,5) circle (2.5pt);
\draw [fill=black] (15,3) circle (2.5pt);
\draw [fill=black] (17,3) circle (2.5pt);
\draw [fill=black] (15,5) circle (2.5pt);
\draw [fill=black] (13,1) circle (2.5pt);
\draw [fill=black] (15,3) circle (2.5pt);
\draw [fill=black] (15,1) circle (2.5pt);
\draw [fill=black] (13,3) circle (2.5pt);
\draw [fill=black] (11,3) circle (2.5pt);
\draw [fill=black] (13,5) circle (2.5pt);
\draw [fill=black] (11,5) circle (2.5pt);
\draw [fill=black] (13,3) circle (2.5pt);
\draw [fill=black] (11,-1) circle (2.5pt);
\draw [fill=black] (13,1) circle (2.5pt);
\draw [fill=black] (11,1) circle (2.5pt);
\draw [fill=black] (13,-1) circle (2.5pt);
\draw [fill=black] (15,3) circle (2.5pt);
\draw [fill=black] (13,1) circle (2.5pt);
\draw [fill=black] (13,3) circle (2.5pt);
\draw [fill=black] (15,1) circle (2.5pt);
\draw [fill=black] (17,1) circle (2.5pt);
\draw [fill=black] (15,-1) circle (2.5pt);
\draw [fill=black] (17,-1) circle (2.5pt);
\draw [fill=black] (15,1) circle (2.5pt);
\draw [fill=black] (11,-1) circle (2.5pt);
\draw [fill=uuuuuu] (16,4) circle (2pt);
\draw [fill=black] (13,5) circle (2.5pt);
\draw [fill=uuuuuu] (14,2) circle (2pt);
\draw [fill=black] (13,5) circle (2.5pt);
\draw [fill=uuuuuu] (12,4) circle (2pt);
\draw [fill=uuuuuu] (12,0) circle (2pt);
\draw [fill=black] (15,-1) circle (2.5pt);
\draw [fill=black] (15,-1) circle (2.5pt);
\draw [fill=uuuuuu] (16,0) circle (2pt);
\draw [fill=black] (17,5) circle (2.5pt);
\draw [fill=black] (17,5) circle (2.5pt);
\draw [fill=black] (11,11) circle (2.5pt);
\draw [fill=black] (13,13) circle (2.5pt);
\draw [fill=black] (11,13) circle (2.5pt);
\draw [fill=black] (13,11) circle (2.5pt);
\draw [fill=black] (15,15) circle (2.5pt);
\draw [fill=black] (13,13) circle (2.5pt);
\draw [fill=black] (13,15) circle (2.5pt);
\draw [fill=black] (15,13) circle (2.5pt);
\draw [fill=black] (17,13) circle (2.5pt);
\draw [fill=black] (15,11) circle (2.5pt);
\draw [fill=black] (17,11) circle (2.5pt);
\draw [fill=black] (15,13) circle (2.5pt);
\draw [fill=black] (17,17) circle (2.5pt);
\draw [fill=black] (15,15) circle (2.5pt);
\draw [fill=black] (17,15) circle (2.5pt);
\draw [fill=black] (15,17) circle (2.5pt);
\draw [fill=black] (13,13) circle (2.5pt);
\draw [fill=black] (15,15) circle (2.5pt);
\draw [fill=black] (15,13) circle (2.5pt);
\draw [fill=black] (13,15) circle (2.5pt);
\draw [fill=black] (11,15) circle (2.5pt);
\draw [fill=black] (13,17) circle (2.5pt);
\draw [fill=black] (11,17) circle (2.5pt);
\draw [fill=black] (13,15) circle (2.5pt);
\draw [fill=black] (17,17) circle (2.5pt);
\draw [fill=uuuuuu] (12,12) circle (2pt);
\draw [fill=black] (15,11) circle (2.5pt);
\draw [fill=uuuuuu] (14,14) circle (2pt);
\draw [fill=black] (15,11) circle (2.5pt);
\draw [fill=uuuuuu] (16,12) circle (2pt);
\draw [fill=uuuuuu] (16,16) circle (2pt);
\draw [fill=black] (13,17) circle (2.5pt);
\draw [fill=black] (13,17) circle (2.5pt);
\draw [fill=uuuuuu] (12,16) circle (2pt);
\draw [fill=black] (11,11) circle (2.5pt);
\draw [fill=black] (11,11) circle (2.5pt);
\draw [fill=black] (17,11) circle (2.5pt);
\draw [fill=black] (5,17) circle (2.5pt);
\draw [fill=black] (3,15) circle (2.5pt);
\draw [fill=black] (5,15) circle (2.5pt);
\draw [fill=black] (3,17) circle (2.5pt);
\draw [fill=black] (1,13) circle (2.5pt);
\draw [fill=black] (3,15) circle (2.5pt);
\draw [fill=black] (3,13) circle (2.5pt);
\draw [fill=black] (1,15) circle (2.5pt);
\draw [fill=black] (-1,15) circle (2.5pt);
\draw [fill=black] (1,17) circle (2.5pt);
\draw [fill=black] (-1,17) circle (2.5pt);
\draw [fill=black] (1,15) circle (2.5pt);
\draw [fill=black] (-1,11) circle (2.5pt);
\draw [fill=black] (1,13) circle (2.5pt);
\draw [fill=black] (-1,13) circle (2.5pt);
\draw [fill=black] (1,11) circle (2.5pt);
\draw [fill=black] (3,15) circle (2.5pt);
\draw [fill=black] (1,13) circle (2.5pt);
\draw [fill=black] (1,15) circle (2.5pt);
\draw [fill=black] (3,13) circle (2.5pt);
\draw [fill=black] (5,13) circle (2.5pt);
\draw [fill=black] (3,11) circle (2.5pt);
\draw [fill=black] (5,11) circle (2.5pt);
\draw [fill=black] (3,13) circle (2.5pt);
\draw [fill=black] (-1,11) circle (2.5pt);
\draw [fill=uuuuuu] (4,16) circle (2pt);
\draw [fill=black] (1,17) circle (2.5pt);
\draw [fill=uuuuuu] (2,14) circle (2pt);
\draw [fill=black] (1,17) circle (2.5pt);
\draw [fill=uuuuuu] (0,16) circle (2pt);
\draw [fill=uuuuuu] (0,12) circle (2pt);
\draw [fill=black] (3,11) circle (2.5pt);
\draw [fill=black] (3,11) circle (2.5pt);
\draw [fill=uuuuuu] (4,12) circle (2pt);
\draw [fill=black] (5,17) circle (2.5pt);
\draw [fill=black] (5,17) circle (2.5pt);
\draw [fill=black] (-1,17) circle (2.5pt);
\draw [fill=black] (5,17) circle (2.5pt);
\draw [fill=black] (5,17) circle (2.5pt);
\draw [fill=black] (5,17) circle (2.5pt);
\draw[color=black] (14.29,-1.5) node {$U_{4k+1}$};
\draw[color=black] (8.2,4.5) node {$U_{4k}$};
\draw[color=black] (14.31,10.5) node {$U_{4k+3}$};
\draw[color=black] (2.31,10.5) node {$U_{4k+2}$};
\draw[color=black] (2.31,-1.5) node {$V^{(k)}$};
\end{scriptsize}
\end{tikzpicture}
\caption{The sets $U_{4k+j}$}\label{fig_U4k}
\end{figure}

Let $n\in\N$; let $z_0$ be the center of $V^{(n)}$, and $\{z_i\}_{i=1,\cdots,4}$ its four boundary points. Define a function $g_n$ on $X$ and supported in $V^{(n)}$ in the following way: $g_n(z_0)=1$, $g_n(z_i)=0$ for all $i=1,\cdots,4$, and $g_n$ is harmonic at every other point of $V^{(n)}$. So, thanks to the description of harmonic functions on the Vicsek cable system (see \cite[Section 3]{DRY}), $g_n$ is constant on the branches emanating from the two diagonals of $V^{(n)}$, and linear on the four half-diagonals of $V^{(n)}$. This function $g_n$ is analogous to the one from \cite[Section 5]{CCFR}. Notice then that $g_n\ge \frac 23$ in $U_{4n}$.  Also, there is a positive constant $c$, independant of $n$, such that
$$
e^{-\Delta} {\bf 1}_{U_{4n}}\ge c\mbox{ in }D_n.
$$
Indeed, whenever $d(x,y)\le 1$,
\begin{eqnarray*}
p_1(x,y) & \ge  &\frac c{V(x,1)}\exp\left(-Cd(x,y)^{\frac{\beta}{\beta-1}}\right)\\
& = & \frac c{V(x,1)}\exp\left(-Cd(x,y)^{\frac{\alpha+1}{\alpha}}\right)\\
& \ge & c^{\prime}>0.
\end{eqnarray*}
Let $D_n=\{x\in U_{4n}\,;\,d(x,(U_{4n})^c)\geq 1\}$. Note that $m(D_n)\simeq m(U_{4n})$ for $n\geq 1$. As a consequence of the above inequality, for all $x\in D_n$,
\begin{eqnarray*}
e^{-\Delta} {\bf 1}_{U_{4n}}(x) & = & \int_{y\in U_{4n}} p_1(x,y)dm(y)\\
& \ge & \int_{d(y,x)\le 1} p_1(x,y)dm(y)\\
& \ge & c^{\prime}V(x,1)\\
& \ge & c^{\prime\prime}.
\end{eqnarray*}
Since $g_n\ge c{\bf 1}_{U_{4n}}$, it follows that
\begin{eqnarray*}
\left\Vert e^{-\Delta}g_n\right\Vert_p & \ge & c \left\Vert e^{-\Delta}{\bf 1}_{U_{4n}}\right\Vert_p\\
& \ge & cm( U_{4n})^{\frac 1p} \\
& \ge & cm(V^{(n)})^{\frac 1p}.
\end{eqnarray*}
We now argue by contradiction: assume that \eqref{eq:quasiRRp} holds for some $\frac{1}{\beta}<\gamma<\frac{\alpha}{\beta}$ and some $1<p<\frac{\alpha-1}{\gamma(\alpha+1)-1}$. Observe that, if $f\ge 0$ and $\left\Vert f\right\Vert_p\le \left\Vert f\right\Vert_1$, one has
$$
\left\Vert e^{-\Delta}f\right\Vert_p\le \left\Vert f\right\Vert_p\le \left\Vert f\right\Vert_1=\left\Vert e^{-\Delta}f\right\Vert_1.
$$
Moreover, since $p>1$, and by definition of $g_n$,
$$||g_n||_p\leq m(V^{(n+1)})^{1/p}\leq \frac{2}{3}m(V^{(n)})\leq ||g_n||_1,$$
if $n$ is large enough, as one can see from $|V^{(n+1)}|^{1/p}\simeq 3^{n/p}$, $|V^{(n)}|\simeq 3^{n}$. So, the above observation gives $||e^{-\Delta}g_n||_p\leq ||e^{-\Delta}g||_1$, and we can use \eqref{Nash} with $f:=e^{-\Delta}g_n$. Together with \eqref{eq:quasiRRp}, it gives

$$\left\Vert f\right\Vert_p^{1+\frac{2\gamma p}{(p-1)\alpha^{\prime}}}\lesssim \left\Vert f\right\Vert_1^{\frac{2\gamma p}{(p-1)\alpha^{\prime}}}\left\Vert \Delta^{\gamma}f\right\Vert_p\lesssim \left\Vert f\right\Vert_1^{\frac{2\gamma p}{(p-1)\alpha^{\prime}}}||\nabla g_n ||_p.$$
Using 

$$||f||_1\leq ||g_n||_1\lesssim m(V^{(n)}),\quad ||f||_p\geq cm(V^{(n)})^{\frac{1}{p}},$$
and

$$||\nabla g_n||^p_p\simeq 3^{-np}\mathrm{diam}(V^{(n)})\simeq 3^{-n(p-1)}\simeq m(V^{(n)})^{-\frac{p-1}{\alpha}},$$
and the fact that $m(V^{(n)})\to \infty$ as $n\to \infty$, by performing elementary computations with the exponents we easily get a contradiction.

\medskip

Finally, we treat the case $\gamma\leq \frac{1}{\beta}$ by interpolation. We rely on the following interpolation inequality (see \cite[Proposition 32]{CRT} for a proof): for every $\lambda,\mu,\gamma\geq 0$, $1<p,q,r<+\infty$ and $0<\theta<1$, such that $\lambda=\theta\gamma+(1-\theta)\mu$, $\frac{1}{r}=\frac{\theta}{p}+\frac{1-\theta}{q}$,

\begin{equation}\label{eq:interpolate}
||\Delta^\lambda g||_r\leq ||\Delta^\gamma g||_p^\theta||\Delta^\mu g||_q^{1-\theta}.
\end{equation}
Let $1<p<+\infty$, we argue by contradiction and assume that \eqref{eq:quasiRRp} holds. Let $\lambda\in (\frac{1}{\alpha+1},\frac{\alpha}{\beta})$ be close enough to $\frac{1}{\alpha+1}$, such that 
\begin{equation} \label{conditionp}
p<\frac{\alpha-1}{\lambda(\alpha+1)-1}.
\end{equation}
Let $\theta\in (0,1)$, and define $\mu=\frac{\lambda-\theta\gamma}{1-\theta}$ and $q=r=p$. If $\theta$ is close enough to $1$, one has $\mu\geq \frac{\alpha}{\beta}\ge \frac 12$, so $(\mathrm{RR}_{p,\mu})$ holds by the positive results already established in Theorem \ref{thm:RRp-subg}. Consequently, applying \eqref{eq:interpolate} to $g=e^{-\Delta}f$, we get $(\mathrm{RR}_{p,\lambda})$, which contradicts the first part of the proof, since condition \eqref{conditionp} holds. Hence, \eqref{eq:quasiRRp} does not hold. This completes the proof of Theorem \ref{thm:RRp-subg}. 
\end{proof}

%
%
%
%
%
%

\begin{proof} [Proof of Proposition \ref{pro:RRmfd}: ]

Start with the Nash inequality:

\begin{equation} \label{eq:Nash2}
\left\Vert f\right\Vert_p^{1+\frac{2\gamma p}{(p-1)D}}\lesssim \left\Vert f\right\Vert_1^{\frac{2\gamma p}{(p-1)D}}\left\Vert \Delta^{\gamma}f\right\Vert_p,
\end{equation}
for functions $f$ satisfying $||f||_p\leq ||f||_1$. The approach is the same as in the proof of Theorem \ref{thm:RRp-subg}. Assume by contradiction that \eqref{eq:quasiRRp} holds for some $\gamma<1/2$ and $p\in (1,+\infty)$. Here we consider a sequence of nonnegative functions $g_n$ which are equal to $1$ on $B(o,n)$, to $0$ on $M\setminus B(o,2n)$ and satisfy $|\nabla g_n|(x)\simeq \frac{1}{n}$ for $x\in B(o,\frac{7}{4}n)\setminus B(o,\frac{5}{4}n)$. A direct computation using $V(o,n)\simeq n^D$ shows that for $n$ large enough, one has $||g_n||_p\leq ||g_n||_1$, hence as in the proof of  Theorem \ref{thm:RRp-subg} one has $||e^{-\Delta} g_n||_p\leq ||g_n||_p\leq ||g_n||_1= ||e^{-\Delta}g_n||_1$. Therefore, applying \eqref{eq:Nash2} to $f_n:=e^{-\Delta}g_n$, by performing elementary computations with the exponents we get a contradiction for $n\to +\infty$ if $\gamma<\frac{1}{2}$, given the validity of the following estimates for $n$ large enough:

$$||e^{-\Delta}g_n||_p\gtrsim n^{D/p},\quad ||e^{-\Delta}g_n||_1\lesssim n^D, $$
and

$$||\nabla g_n||_p\lesssim n^{\frac{D}{p}-1}.$$
The above estimate of $||\nabla g_n||_p$ is easily obtained; for $||e^{-\Delta}g_n||_1$, we use the equality $||e^{-\Delta}g_n||_1=||g_n||_1$, which is easily seen to be $\lesssim n^D$. It remains to prove the estimate for $n$ large enough:

\begin{equation}\label{eq:Lp-heat}
||e^{-\Delta}g_n||_p\gtrsim n^{D/p}
\end{equation}
Let $A>0$, the heat kernel upper-bound implies that for every $x\in M\setminus B(o,2(A+1)n)$,

\begin{eqnarray*}
|e^{-\Delta}g_n(x)| &\leq & C e^{ -cd^2(x,B(o,2n))} \int_{B(o,2n)}g_n\\
&\leq & C' e^{ -2cA^2n^2} n^D e^{ -cd^2(x,B(o,2n))/2}.
\end{eqnarray*}
One take $A$ large enough such that $C' e^{ -2cA^2n^2}n^D<1$ for all $n$. Fix such an $A>0$, then one obtains

$$|e^{-\Delta}g_n(x)| \leq  e^{ -cd^2(x,B(o,2n))/2}.$$
Given the assumption on the volume growth of $M$, one obtains

$$\int_{M\setminus B(o,2(A+1)n)} |e^{-\Delta}g_n(x)|\,dx\leq C,$$
for some constant $C>0$ depending only on the implicit constants in the volume growth of balls. We claim that 

$$||e^{-\Delta}g_n||_p\geq \left(\int_{B(o,2(A+1)n)} |e^{-\Delta}g_n(x)|^p\,dx\right)^{1/p} \gtrsim n^{D/p},$$
which will conclude the argument. To see this, one has by H\"older,

$$\left(\int_{B(o,2(A+1)n)} |e^{-\Delta}g_n(x)|^p\,dx\right)^{1/p} \geq  n^{D(\frac{1}{p}-1)}\left(\int_{B(o,2(A+1)n)} |e^{-\Delta}g_n(x)|\,dx\right).$$
But

\begin{eqnarray*}
\int_{B(o,2(A+1)n)} |e^{-\Delta}g_n(x)|\,dx & \Mg = \Bk & ||e^{-\Delta}g_n||_1-\int_{M\setminus B(o,2(A+1)n)} |e^{-\Delta}g_n(x)|\,dx\\
&\geq & C'n^D-C\\
&\gtrsim & n^D,
\end{eqnarray*} 
if $n$ is large enough, hence

$$\left(\int_{B(o,2(A+1)n)} |e^{-\Delta}g_n(x)|^p\,dx\right)^{1/p} \gtrsim n^{D/p},$$
which was to be shown.

\end{proof}

\begin{center}{\bf Acknowledgements} \end{center}
This work was partly supported by the French ANR project
RAGE ANR-18-CE40-0012.

\bibliographystyle{plain}               
\bibliography{reverseriesz}

\end{document}